\documentclass[aop,MSNbibl,seceqn,dvips]{arximspdf}

\doi{10.1214/12-AOP825} 
\volume{42}
\issue{1}
\pubyear{2014}
\firstpage{168}
\lastpage{203}

\makeatletter

\newproclaim{Remark}{Remark}

\newcommand{\rrvert}{\vert}
\newcommand{\llvert}{\vert}

\newcommand{\cal}{\mathcal}

\newtheorem{theorem}{Theorem}[section]
\newtheorem{lemma}[theorem]{Lemma}
\newtheorem{proposition}[theorem]{Proposition}

\def\R{{\mathbb R}}
\def\E{{{\mathbb E}}}
\def\N{{\mathbb N}}
\def\si{\sigma}
\def\Var{\operatorname{Var}}
\def\Om{{\Omega}}
\def\cF{{\cal F}}
\def\be{{\beta}}
\def\de{{\delta}}

\makeatother

\begin{document}
\begin{frontmatter}

\title{Central limit theorem for an additive functional of the fractional Brownian motion}
\runtitle{CLT for an additive functional of the fBm}

\begin{aug}
\author[A]{\fnms{Yaozhong} \snm{Hu}\thanksref{t1}\ead[label=e1]{hu@math.ku.edu}},
\author[A]{\fnms{David} \snm{Nualart}\thanksref{t2}\ead[label=e2]{nualart@math.ku.edu}}
\and
\author[B]{\fnms{Fangjun} \snm{Xu}\corref{}\thanksref{t3}\ead[label=e3]{fangjunxu@gmail.com}}
\runauthor{Y. Hu, D. Nualart and F. Xu}
\affiliation{University of Kansas, University of Kansas and\break  East China Normal University}
\address[A]{Y. Hu\\
D. Nualart\\
Department of Mathematics\\
University of Kansas\\
Lawrence, Kansas 66045\\
USA\\
\printead{e1}\\
\hphantom{E-mail: }\printead*{e2}} 
\address[B]{F. Xu\\
School of Finance and Statistics\\
East China Normal University\\
Shanghai, China 200241\\
\printead{e3}}
\end{aug}

\thankstext{t1}{Supported in part by a grant from the Simons
Foundation \#209206.}
\thankstext{t2}{Supported by NSF Grant DMS-12-08625.}
\thankstext{t3}{Supported in part by the Robert Adams Fund.}

\received{\smonth{11} \syear{2011}}
\revised{\smonth{8} \syear{2012}}

%
\begin{abstract}
We prove a central limit theorem for an additive functional of the
$d$-dimensional fractional Brownian motion with Hurst index
$H\in(\frac{1}{1+d},\frac{1}{d})$, using the method of moments,
extending the result by Papanicolaou, Stroock and Varadhan in the case
of the standard Brownian motion.
\end{abstract}

%

\end{frontmatter}

\section{Introduction}\label{sec1}
Let $ \{ B(t)=(B^1(t),\ldots,
B^d(t)), t\geq0 \}$ be a $d$-dimen-\break sional fractional Brownian
motion (fBm) with Hurst index $H\in(0,1)$. If $Hd<1$, then the local
time of $B$ exists (see, e.g., \cite{GH,HN,HNS}) and can be
defined as
\[
L_t(x)=\int_0^t \delta
\bigl(B(s)-x\bigr)\,ds,\qquad t\ge0, x\in\R^d,
\]
where $\de$ is the Dirac delta function. The above local time is jointly
continuous with respect to $t$ and $x$; see \cite{GH}. For any
integrable function $f\dvtx \R^d \rightarrow\R$, one can easily show the
following convergence in law in the space $C([0,\infty))$,
as $n$ tends to infinity:
%
\begin{equation}
\label{e11} \biggl( n^{Hd-1} \int_0^{nt}
f\bigl(B(s)\bigr)\,ds, t\ge0 \biggr) \stackrel{\mathcal{L}} {\rightarrow}
\biggl( L_t(0) \int_{\R^d} f(x)\,dx, t\ge0 \biggr).
\end{equation}
In fact, making the change of variable $s=nu$, and using the
scaling property of the fBm, we see that the process $ (
n^{Hd-1} \int_0^{nt} f(B(s))\,ds, t\ge0 ) $ has the same law
as
\begin{eqnarray*}
n^{Hd} \int_0^t f
\bigl(n^HB(u)\bigr)\,du &=& n^{Hd} \int_{\R^d}
f\bigl(n^Hx\bigr) L_t(x)\,dx
\\
&=& \int_{\R^d} f( x) L_t\bigl(n^{-H}x
\bigr)\,dx,\qquad t\ge0.
\end{eqnarray*}
From here it is straightforward to verify (\ref{e11}).

If we assume that $\int_{\R^d}f(x)\,dx=0$, then we see $n^{Hd-1} \int_0^{nt} f(B(s))\,ds$ converges to $0$. It is interesting to know if
there is a $\beta>Hd-1$ such that $n^{\beta} \int_0^{nt} f(B(s))\,ds$ converges to a nonzero process. This will be proved to be true.
In order to formulate this result we introduce the
following space of functions. Fix a number $\beta>0$, and denote
\[
H_{0}^{\beta}= \biggl\{f \in L^{1}\bigl(
\R^d\bigr)\dvtx  \int_{\R^d}\bigl|f(x)\bigr| |x|^{\beta}
\,dx<\infty \mbox{ and } \int_{\R^d }f(x)\,dx=0 \biggr\}.
\]
For any $f\in H_{0}^{\beta}$, and assuming $\beta\in(0,2)$, by Lemma
\ref{lema1}, the quantity
%
\begin{equation}
\label{beta} \Vert f\Vert_{\beta}^{2}:=-\int
_{\R^{2d}}f(x)f(y)|x-y|^{\beta}\,dx\,dy
\end{equation}
is finite and nonnegative. The next theorem is the main result of this paper.

\begin{theorem}
\label{th1} Suppose $\frac{1}{d+1}<H<\frac{1}{d}$ and $f\in H^{{1}/{H}-d}_0$.
Then
\[
\biggl( n^{({Hd-1})/{2}}\int_{0}^{nt} f
\bigl(B(s)\bigr)\,ds, t\ge0 \biggr) \stackrel{\mathcal{L}} {
\longrightarrow} \Bigl( \sqrt{C_{H,d}} \| f\|_{{1}/{H}-d} W
\bigl(L_{t}(0)\bigr), t\ge 0 \Bigr)
\]
in the space $C([0,\infty))$, as $n$ tends to infinity, where
``$\stackrel{\mathcal{L}}{\longrightarrow}$'' denotes the convergence
in law, $W $ is a real-valued standard Brownian motion independent of
$B$ and
\begin{eqnarray*}
C_{H,d}& = &\frac2 {(2\pi)^{d/2}} \int
_0^\infty w^{-Hd} \biggl(1-\exp\biggl(-
\frac{1}{2w^{2H}}\biggr) \biggr) \,dw
\\
&=& \frac{2^{1-{1}/({2H})}}{(1-Hd)\pi^{d/2}}\Gamma
\biggl(\frac{Hd+2H-1}{2H} \biggr).
\end{eqnarray*}
\end{theorem}
Notice\vspace*{1pt} that $\frac{Hd-1}{2}>Hd-1$ since $H<\frac{1}{d}$. When $d=1$
and $H=\frac12$, the above theorem is
obtained by Papanicolaou, Stroock and Varadhan in
\cite{PSV} with \mbox{$C_{1/2, 1}=2$}. On the other hand, the constant
$C_{H,d}$ is finite for any
$H>\frac{1}{d+2}$. We conjecture that our result also holds for $\frac
{1}{d+2}<H<\frac{1}{d}$. But we have not been able to show our
result in the case $H\leq\frac{1}{d+1}$. The main reason is that in
the proof of Proposition \ref{prop4} we need $H>\frac{1}{d+1}$; see
the \hyperref[Rem]{remark} at the end of Section
\ref{sec3}.\vspace*{1pt}\vadjust{\goodbreak}

In the critical case $Hd=1$, the local time does not exist. For the
Brownian motion case ($H=\frac12$ and $d=2$), Kallianpur and Robbins
\cite{KaRo} proved that for any bounded and integrable function
$f\dvtx \mathbb{R}^2\rightarrow\mathbb{R}$,
\[
\frac1{\log n} \int_0^n f(B_s)\,ds
\stackrel{\mathcal {L}} {\longrightarrow} \frac{Z}{2\pi} \int
_{\mathbb{R}^2 }f(x)\,dx
\]
as $n$ tends to infinity, where $Z$ is a random variable with
exponential distribution of parameter $1$. A functional version of this
result was given
by Kasahara and Kotani in \cite{KaKo}, where they also proved the
second-order results when $\int_{\R^2 }f(x)\,dx=0$. The
Kallianpur--Robbins law was extended to the fBm by K\^ono in \cite
{Kon}, and the corresponding functional version was obtained by
Kasahara and Kosugi in \cite{KaKos}. However, second-order results for
the fBm in the critical case $Hd=1$ have not been yet proved. On the
other hand, we refer to Biane \cite{biane} for some extensions of
these results to the case of functionals of $k$ independent Brownian motions.

The limit theorem proved in this paper, where an independent source of
noise appears in the limit, might be connected to the central limit
theorems for weighted power variations of the fractional Brownian
motion in the critical cases $H=\frac14$ and $H=\frac16$ (see the
works by Nourdin and R\'eveillac \cite{NoRe} and Nourdin, R\'eveillac
and Swanson \cite{NoReSw}), where a similar phenomenon happens.
However, the method of proof in these papers relies on the techniques
of Malliavin calculus, and it is related to a general central limit
theorem for multiple Skorohod integrals obtained by Nourdin and Nualart
in \cite{NoNu}. These techniques do not seem to work for the class of
additive functionals considered in this paper.

We would like to give a heuristic explanation of the fact that an
independent noise appears in the limit, and also to indicate the main
ideas of the proof. First, by the scaling property of the fractional
Brownian motion, we can consider the continuous process
%
\begin{equation}
\label{Fn} F_{n}(t):=n^{({1+Hd})/{2}}\int_{0}^{t}
f\bigl(n^{H}B(s)\bigr)\,ds.
\end{equation}
Making the change of variables $u_1=n(s_2-s_1)$ and $u_2=s_2$, we can
write formally
\begin{eqnarray*}
F_{n}(t)^2 &=& 2n^{1+Hd}\int_{0}^{t}
\int_{0}^{s_2} f\bigl(n^{H}B(s_1)
\bigr)f\bigl(n^{H}B(s_2)\bigr) \,ds_1\,ds_2
\\
&=& 2n^{Hd} \int_{0}^{t} \int
_{0}^{n u_2} f\biggl(n^{H}B
\biggl(u_2- \frac
{u_1}n\biggr)\biggr)f\bigl(n^{H}B(u_2)
\bigr) \,du_1\,du_2
\\
&=& 2n^{ Hd} \int_{0}^{t} \int
_{0}^{n u_2} \int_{\mathbb{R}^d} f
\biggl(n^{H} \biggl(B\biggl(u_2- \frac{u_1}n\biggr)
-B(u_2) \biggr) + n^Hz \biggr)
\\
&&\hspace*{81.2pt}{} \times f\bigl(n^H z\bigr) \delta\bigl(B(u_2)-z\bigr)
\,dz\,du_1 \,du_2.
\end{eqnarray*}
The change of variable $n^Hz=x$ yields
%
\begin{eqnarray}
\label{u1} F_{n}(t)^2
&=& 2 \int_{0}^{t} \int_{0}^{n u_2}
\int_{\mathbb{R}^{d}} f\bigl(B_n(u_1,u_2)
+ x\bigr) \nonumber\\[-8pt]\\[-8pt]
&&\hspace*{62.3pt}{}\times f(x) \delta\biggl(B(u_2)-\frac{x}{n^H}\biggr) \,dx\,du_1
\,du_2,
\nonumber
\end{eqnarray}
where $B_n(u_1,u_2)= n^{H} (B(u_2- \frac{u_1}n) -B(u_2) )$.
Notice that $B_n(u_1,u_2)$ is a $d$-dimensional centered Gaussian
vector whose components are independent and with variance $u_1^{2H}$.
Using the covariance function of the fractional Brownian motion, it is
easy to show that for any $u_1, u_2, u_3 \ge0$, $(B_n(u_1,u_2),
B(u_3))$ (assuming $u_1 \le nu_2$) converges in law to
\[
\bigl(B_\infty(u_1, u_2), B(u_3)
\bigr),
\]
where $B_\infty(u_1, u_2)$ is independent of $B(u_3)$. As a consequence,
\begin{eqnarray*}
&& \lim_{n\rightarrow\infty} \mathbb{E} \biggl( f\bigl(B_n(u_1,u_2)
+ x\bigr) \delta\biggl(B(u_2)-\frac{x}{n^H}\biggr) \biggr)
\\
&&\qquad = \mathbb{E} \bigl( f\bigl(B_\infty(u_1,u_2)
+ x\bigr) \bigr) \mathbb{E} \bigl(\delta\bigl(B(u_2) \bigr)\bigr).
\end{eqnarray*}
Assuming that we can commute the expectation with the integrals, this
formally yields
\begin{eqnarray*}
\lim_{n\rightarrow\infty} \mathbb{E} \bigl(F_n(t)^2
\bigr)&=&\frac2{(2\pi )^{d/2}} \E \biggl(\int_0^t
\delta\bigl(B(u_2)\bigr) \,du_2 \biggr)
\\
&&{} \times\int_0^\infty\int_{\R^{2d}}
f(x+y)f(x) \\
&&\hspace*{18.5pt}\qquad\quad{}\times u^{-Hd}_1 \bigl(
e^{-{|y|^2}/({2u_1^{2H}})} -1 \bigr)\,dy \,dx\,du_1,
\end{eqnarray*}
where we can add the term $-1$ because the integral of $f$ is zero. The
right-hand side of the above expression
is equal to
\begin{eqnarray*}
&& \frac2{(2\pi)^{d/2}} \E \bigl(W\bigl(L_t(0)
\bigr)^2 \bigr) \\
&&\quad{}\times\int_0^\infty\int
_{\R^{2d}} f(x+y)f(x) u^{-Hd}_1\bigl(
e^{-{|y|^2}/({2u_1^{2H}})} -1\bigr)\,dy\,dx\,du_1
\\
&&\qquad= \frac2{(2\pi)^{d/2}} \E \bigl(W\bigl(L_t(0)
\bigr)^2 \bigr)
\\
&&\qquad\quad{} \times\int_0^\infty w^{-Hd}\bigl(1-
e^{-{1}/({2w^{2H}})}\bigr)\,dw \int_{\R^{2d}} -f(x+y)f(x)
|y|^{{1}/{H}-d}\,dy\,dx
\\
&&\qquad=C_{H,d} \| f\|^2_{{1}/{H}-d} \E \bigl(W
\bigl(L_t(0)\bigr)^2 \bigr),
\end{eqnarray*}
and this shows the convergence of the moments of order two.
Roughly speaking, the term $B(u_2)$ appearing in (\ref{u1})
contributes to the local time at zero whereas $B_n(u_1,u_2)$ becomes
independent of $B$ in the limit and contributes to the constant $C_{H,d}$.
The main technical difficulty is the commutation of the expectation
with the limit, and for this we will implement a convenient truncation
argument. The idea is to replace the interval $[0,nu_2]$ by a compact
set $[0,K]$, and then show that the integral over $[K, nu_2]$ converges
to zero as $K$ tends to infinity, uniformly in~$n$.
However, this convergence holds only if we integrate over $[K, \frac
{nu_2}2]$, and for this reason, we need to show (see Proposition \ref{prop3})
that the integral over $[\frac{nu_2}2, nu_2]$ tends to zero as $n$
tends to infinity.
We will make this heuristic computation rigorous when we compute the
limit of the moments of even order of a vector of increments of the
process $F_n(t)$.
In this case, we will have an even product of factors, and for each
couple of consecutive factors, we will make the above change of variables.

The basic idea of the approach used in this paper is to apply the
method of moments to an additive functional, and when dealing with an
integral on $[0,t]^{2m}$, with respect to the measure $ds_1 \cdots
ds_{2m}$, we make the change of variables
$u_{2k-1}=n(s_{2k}-s_{2k-1})$ and $u_{2k}=s_{2k}$, $1\le k\le m$. Then
the increments of $B$ in small intervals will be responsible for the
independent noise appearing in the limit. This methodology could be
applied to other examples of additive functionals and processes.

After some preliminaries in Section \ref{sec2}, Section \ref{sec3} is devoted to the
proof of Theorem~\ref{th1}, based on the method of moments.
Throughout this paper, if not mentioned otherwise, the letter $c$,
with or without a subscript, denotes a generic positive finite
constant whose exact value is independent of $n$ and may change from
line to line. We use $\iota$ to denote $\sqrt{-1}$.

\section{Preliminaries}\label{sec2}
Let $ \{ B(t)= (B^1(t),\ldots,
B^d(t)), t\geq0 \} $ be a $d$-dime-\break nsional fractional Brownian
motion with Hurst index $H\in(0,1)$, defined on some probability space
$(\Om, \cF, P)$. That is, the components
of $B$ are independent centered Gaussian processes with covariance
\[
\E \bigl(B^i(t)B^i(s) \bigr)=\tfrac{1}{2}
\bigl(t^{2H}+s^{2H}-|t-s|^{2H} \bigr).
\]

We refer to \cite{BHOZ} for a detailed analysis of this process.

The next lemma gives a formula for the moments of the increments of the process
$ \{W (L_{t}(0))\dvtx  t\ge0 \}$ on disjoint intervals, where
$W$ is a real-valued standard Brownian motion
independent of $B$.
%
\begin{lemma} \label{lema2} Fix a finite number of disjoint intervals
$(a_i, b_i]$ in $[0,\infty)$, where $i = 1, \ldots, N$ and $b_i \leq a_{i+1}$.
Consider a multi-index $\mathbf{m} = (m_1, \ldots, m_N)$, where $m_i
\geq1$ and
$1 \leq i \leq N$. Then
%
\begin{equation}
\label{21} \E \Biggl(\prod_{i=1}^N
\bigl[W \bigl(L_{b_i} (0) \bigr) - W \bigl(L_{a_i} (0)\bigr)
\bigr]^{m_i} \Biggr)
\end{equation}
is equal to
\[
\Biggl(\prod_{i=1}^N \frac{m_i!}{ 2^{{m_i}/{2}} (2\pi
)^{{m_id}/{4} } (m_i/2)!}
\Biggr) \int_{\prod_{i=1}^N [a_i,b_i]^{{m_i}/{ 2}}} \det
\bigl(A(w)\bigr)^{-1/2}\,dw,
\]
if all $m_i$ are even and $0$ otherwise, where $A(w)$ is the covariance
matrix of the Gaussian random vector
\[
\biggl(B\bigl(w^i_k\bigr)\dvtx  1\leq i\leq N \mbox{ and } 1
\leq k\leq\frac
{m_i}{2} \biggr).
\]
\end{lemma}

\begin{pf}
It is easy to see that when one of $m_i$ is odd, the expectation is
$0$. Suppose now that all $m_i$ are even. Denote by $\cF^B$ the $\si
$-algebra generated by the fractional Brownian motion $B$. Since $W$ is
a standard Brownian motion independent of $B$, we have
\begin{eqnarray*}
&&\E\prod_{i=1}^N \bigl[W
\bigl(L_{b_i} (0) \bigr) - W \bigl(L_{a_i} (0)\bigr)
\bigr]^{m_i}
\\[-0.7pt]
&&\qquad= \E \Biggl\{ \E \Biggl( \prod_{i=1}^N
\bigl[W \bigl(L_{b_i} (0) \bigr) - W \bigl(L_{a_i} (0)\bigr)
\bigr]^{m_i} \Big|\cF^B \Biggr) \Biggr\}
\\[-0.7pt]
&&\qquad= \Biggl[ \prod_{i=1}^N
\frac{m_i!}{ 2^{{m_i}/{2}} (m_i/2)!} \Biggr] \E\prod_{i=1}^N
\bigl[ L_{b_i} (0) - L_{a_i} (0) \bigr]^{m_i/2}
\\[-0.7pt]
&&\qquad= \Biggl(\prod_{i=1}^N
\frac{m_i!}{ 2^{{m_i}/{2}} (2\pi
)^{{m_id}/{4} } (m_i/2)!} \Biggr) \int_{\prod_{i=1}^N [a_i,b_i]^{m_i/ 2}} \det \bigl(A(w)
\bigr)^{-1/2} \,dw.
\end{eqnarray*}
This completes the proof.
\end{pf}

We shall use the following local nondeterminism property of the fractional
Brownian motion (see \cite{Be}): for any $n\ge2$ there exists a
positive constant $k_H$ depending on $n$, such that for any
$0=s_0<s_1\leq\cdots\leq
s_n<\infty$ and $ u_1,\ldots, u_n\in\R^d$,
%
\begin{equation}\label{lndp}
\Var \Biggl(\sum^n_{i=1} u_i
\cdot \bigl(B(s_i)-B(s_{i-1}) \bigr) \Biggr)\geq
k_H\sum^n_{i=1}|u_i|^2(s_i-s_{i-1})^{2H}.
\end{equation}
This can also be written as
%
\begin{equation}\label{modified-lndp}
\Var \Biggl(\sum^n_{i=1} u_i
\cdot B(s_i) \Biggr)\geq k_H\sum
^n_{i=1} \Biggl|\sum_{j=i}^n
u_j \Biggr| ^2(s_i-s_{i-1})^{2H}.
\end{equation}

We claim that the law of the random vector $ (W (L_{b_i} (0) ) - W
(L_{a_i} (0))\dvtx  1\le i\le N )$ is determined by
the moments computed in Lemma \ref{lema2}. This is a consequence of
the following estimates. Fix an even integer $n=2k$, and set
$D_k= \{ s\in[0,t]^{k}\dvtx  0< s_1< s_2< \cdots< s_k< t  \}$. Let
$A_k(s)$ be the covariance matrix of Gaussian random vector $
(B(s_1), B(s_2), \ldots, B(s_k) )$. Then the local nondeterminism
property (\ref{lndp})
implies that
%
\begin{equation}\label{det}
\bigl(\det A_k(s) \bigr)^{-{1}/{2}} \le c^k \prod
^k_{i=1}(s_i-s_{i-1})^{-Hd}
\qquad\mbox{for some constant } c.
\end{equation}
As a consequence of (\ref{21}) and (\ref{det}),
\begin{eqnarray*}
\E \bigl[W \bigl(L_{t}(0)\bigr) \bigr]^n & \le&
c^k n! \int_{D_k}s_1^{-Hd}
(s_2-s_1)^{-Hd} \cdots (s_k-s_{k-1})^{-Hd}
\,ds
\\[-0.7pt]
& = & c^k n! \int_{\{0<u_1 + \cdots+ u_k <t\}} \prod
_{i=1}^k u_i^{-Hd}\,du
\\[-0.7pt]
& = & c^k n! t^{k(1-Hd)}\frac
{\Gamma^{k}(1-Hd)}{\Gamma (k(1-Hd)+1 )}.
\end{eqnarray*}
Therefore, $\E [W (L_{t}(0)) ]^n $ is bounded by $c^k n!/
\Gamma (k(1-Hd)+1 )$, and this easily implies the desired
characterization of the law of the increments of the process
$ \{W (L_{t}(0))\dvtx  t\ge0 \}$ on disjoint intervals by its moments.

\section{\texorpdfstring{Proof of Theorem \protect\ref{th1}}
{Proof of Theorem 1.1}}\label{sec3}

By the scaling property of the fractional Brownian motion we
see that, as processes indexed by $t\ge0$,
\[
n^{({Hd-1})/{2}}\int_{0}^{nt} f\bigl(B(s)\bigr)
\,ds \stackrel{\mathcal{L}} {=} n^{({1+Hd})/{2}}\int_{0}^{t}
f\bigl(n^{H}B(s)\bigr)\,ds.
\]
Therefore, it suffices to show
the theorem for the continuous process defined in (\ref{Fn}).
The proof of Theorem \ref{th1} will be done in two steps. We first
show tightness, and then establish the convergence of moments.
Tightness will be deduced from the following result.

\begin{proposition} \label{prop1}
For any $0\leq a<b \leq t$ and any integer $m\geq1$,
\begin{eqnarray*}
&&
\E \bigl[ \bigl(F_{n}(b)-F_{n}(a) \bigr)^{2m}
\bigr] \\
&&\qquad\leq C \biggl((b-a)^{1-Hd} \int_{\R^{2d}}\bigl|f(x)f(y)\bigr|
|y|^{{1}/{H}-d} \,dx\,dy \biggr)^{m},
\end{eqnarray*}
where $C$ is a constant depending only on $H$ and $m$.
\end{proposition}

\begin{pf} Define
%
\begin{equation}\label{e310}
D= \bigl\{s\in\R^{2m}\dvtx  a<s_{1}<s_2<\cdots
<s_{2m}<b \bigr\}.
\end{equation}

Using the following identity for $f\in H^{{1}/{H}-d}_0$
\[ 
f(x)=\frac{1}{(2\pi)^d}\int_{\R^d}e^{\iota x\cdot\xi}\int
_{\R
^d}e^{-\iota\xi\cdot y} f(y)\,dy\,d\xi,
\]
we then have
%
\begin{eqnarray}\label{i32}
&& \E \bigl[ \bigl(F_{n}(b)-F_{n}(a)
\bigr)^{2m} \bigr]
\nonumber\\
&&\qquad=  (2m)!n^{m(1+Hd)} \E \Biggl(\int_{D}\prod
_{i=1}^{2m} f\bigl(n^{H}B(s_i)
\bigr)\,ds \Biggr)
\nonumber
\\
&&\qquad= \frac{(2m)!}{(2\pi)^{2md}} n^{m(1-Hd)}\nonumber\\
&&\qquad\quad{} \times \int_{\R^{2md}} \int
_{D}\int_{\R^{2md}}\prod
_{i=1}^{2m} f(y_i)
\exp \Biggl(-\frac{1}{2}\Var \Biggl(\sum
^{2m}_{i=1}\xi _i \cdot
B(s_i) \Biggr) \nonumber\\
&&\hspace*{219pt}{}- \iota\sum^{2m}_{i=1}
\frac{y_i\cdot\xi
_i}{n^H} \Biggr)\,d\xi\,ds\,dy
\\
&&\qquad= \frac{(2m)!}{(2\pi)^{2md}} n^{m(1-Hd)}
\nonumber
\\
&&\qquad\quad{} \times\int_{\R^{2md}} \int
_{D}\int_{\R^{2md}} \prod
_{i=1}^{2m} f(y_i)\exp \Biggl(-\frac{1}{2}\Var \Biggl(\sum
^{2m}_{i=1}\xi _i \cdot
B(s_i) \Biggr) \Biggr) \nonumber\\
&&
\hspace*{93pt}\qquad\quad{} \times
\prod^{2m}_{i=1}
\bigl(e^{-\iota
{y_i\cdot\xi_i}/{n^H}}-1 \bigr)\,d\xi\,ds\,dy,
\nonumber
\end{eqnarray}
where in the last equality we used the fact that
\[
\int_{\R^d} f(x)\,dx=0.
\]

By the local nondeterminism property (\ref{modified-lndp}), with the
convention $s_0=0$ and $\eta_{2m+1}=0$, we can write
\begin{eqnarray*}
\hspace*{-4pt}&& \E \bigl[ \bigl(F_{n}(b)-F_{n}(a)
\bigr)^{2m} \bigr]
\\
\hspace*{-4pt}&&\qquad\leq c_1 n^{m(1-Hd)}\\
\hspace*{-4pt}&&\qquad\quad\hspace*{0pt}{}\times \int_{\R^{2md}} \int
_{D} \Biggl|\prod_{i=1}^{2m}
f(y_i) \Biggr| \int_{\R^{2md}} \exp \Biggl(-
\frac{\kappa_H}{2}\sum^{2m}_{i=1} \Biggl|\sum
^{2m}_{j=i}\xi _j\Biggr|^2(s_i-s_{i-1})^{2H}
\Biggr)
\\
\hspace*{-4pt}&&\hspace*{124pt}\qquad\quad{} \times\prod^{2m}_{i=1}
\bigl|e^{-\iota{y_i\cdot
\xi_i}/{n^H}}-1 \bigr|\,d\xi\,ds\,dy
\\
\hspace*{-4pt}&&\qquad= c_1 n^{m(1-Hd)} \\
\hspace*{-4pt}&&\qquad\quad{}\times\int_{\R^{2md}} \int
_{D} \Biggl|\prod_{i=1}^{2m}
f(y_i) \Biggr| \int_{\R^{2md}} \exp \Biggl(-
\frac{\kappa_H}{2}\sum^{2m}_{i=1} |\eta
_i|^2(s_i-s_{i-1})^{2H}
\Biggr)
\\
\hspace*{-4pt}&&\hspace*{124pt}\qquad\quad{} \times\prod^{2m}_{i=1} \biggl|\exp \biggl(
\iota\frac
{y_i}{n^H}\cdot(\eta_{i+1}-\eta_i) \biggr)-1 \biggr|\,d\eta\,ds\,dy,
\end{eqnarray*}
where we made the change of variables
\[
\eta_i=\sum^{2m}_{j=i}\xi_j
\]
for $i= 1,\ldots, 2m$ in the last equality.

Let
\[
x_i=\eta_i (s_i-s_{i-1})^{H}\quad \mbox{and}\quad
u_i=s_i-s_{i-1}
\]
for $i= 1,\ldots, 2m$. Then
\begin{eqnarray*}
&& \E \bigl[ \bigl(F_{n}(b)-F_{n}(a)
\bigr)^{2m} \bigr]
\\
&&\qquad\leq c_1 n^{m(1-Hd)} \\
&&\qquad\quad{}\times\int_{\R^{2md}} \int
_{[a, b]\times
[0,b-a]^{2m-1}} \int_{\R^{2md}}\prod
_{i=1}^{2m} \bigl|f(y_i) \bigr| \Biggl(\prod
_{i=1}^{2m} u^{-Hd}_i
\Biggr)
\\
&&\hspace*{139pt}\qquad\quad{} \times\exp \Biggl(-\frac{\kappa_H}{2}\sum^{2m}_{i=1}
|x_i|^2 \Biggr)\\
&&\hspace*{104pt}\qquad\quad{} \times\prod^{2m}_{i=1}
\biggl|\exp \biggl(\iota\frac{y_i}{n^H}\cdot \biggl(\frac{x_{i+1}}{u^H_{i+1}}-
\frac{x_i}{u^H_i} \biggr) \biggr)-1 \biggr|\,dx \,du\,dy,
\end{eqnarray*}
where $x_{2m+1}=0$ and the integral on $[a, b]\times[0,b-a]^{2m-1}$
means that $u_1\in
[a, b]$ and $u_i\in[0, b-a]$ for $i=2, \ldots, 2m$.

Let
$\sqrt{\kappa_H}X_1, \ldots, \sqrt{\kappa_H}X_{2m}$ be independent
copies of a $d$-dimensional standard normal random vector and
$X_{2m+1}=0$. Then the above inequality can be rewritten as
%
\begin{eqnarray}\label{i33}
&& \E \bigl[ \bigl(F_{n}(b)-F_{n}(a)
\bigr)^{2m} \bigr]
\nonumber\\
&&\quad\le c_2 n^{m(1-Hd)} \nonumber\\[-4pt]\\[-12pt]
&&\qquad{}\times\int_{\R^{2md}} \int
_{[a, b]\times
[0,b-a]^{2m-1} } \Biggl(\prod_{i=1}^{2m}
\bigl| f(y_i) \bigr| u^{-Hd}_i \Biggr)
\nonumber
\\
&&\hspace*{90pt}\qquad{} \times\E \Biggl(\prod^{2m}_{i=1} \biggl|\exp \biggl(\iota\frac{y_i \cdot
X_{i+1}}{n^H u^H_{i+1}}-\iota\frac{y_i \cdot X_i}{n^Hu^H_i} \biggr)-1 \biggr|
\Biggr)\,du\,dy. \nonumber\hspace*{-15pt}
\end{eqnarray}
Notice that
\begin{eqnarray*}
&&
\bigl|e^{(\iota{y_{i}}/{n^H})\cdot
({X_{i+1}}/{u^H_{i+1}}-{X_{i}}/{u^H_{i}} )}-1 \bigr|\\
&&\qquad\leq2\wedge \bigl( \bigl|e^{\iota{y_{i}\cdot
X_{i+1}}/({n^Hu^H_{i+1}})}-1 \bigr| +
\bigl|e^{\iota{y_{i}\cdot
X_{i}}/({n^Hu^H_{i}})}-1 \bigr| \bigr).
\end{eqnarray*}
For each factor in the product inside the expectation in (\ref
{i33}), we choose the upper bound 2 when $i$ is even and the upper bound
\[
\bigl|e^{\iota{y_{i}\cdot X_{i+1}}/({n^Hu^H_{i+1}})}-1 \bigr| + \bigl|e^{\iota{y_{i}\cdot
X_{i}}/({n^Hu^H_{i}})}-1 \bigr|,\
\]
when $i$ is odd. Thus we have
\begin{eqnarray*}
\hspace*{-4pt}&& \E \bigl[ \bigl(F_{n}(b)-F_{n}(a)
\bigr)^{2m} \bigr]
\\
\hspace*{-4pt}&&\qquad\leq c_3 \int_{\R^{2md}} \int_{[a, b]\times[0,b-a]^{2m-1}
}
\E \Biggl[\prod_{i=1}^{m}
\bigl(n^{1-Hd} \bigl|f(y_{2i-1})f(y_{2i}) \bigr|
\bigl(u^{-Hd}_{2i-1}u^{-Hd}_{2i} \bigr)
\\
\hspace*{-4pt}&&\hspace*{34pt}\hspace*{104pt}\qquad\quad{}
\times \bigl( \bigl|e^{\iota{y_{2i-1}\cdot
X_{2i}}/({n^Hu^H_{2i}})}-1 \bigr| \\
\hspace*{-4pt}&&\qquad\quad\hspace*{153pt}{}+ \bigl|e^{\iota{y_{2i-1}\cdot
X_{2i-1}}/({n^Hu^H_{2i-1}})}-1 \bigr| \bigr) \bigr)
\Biggr]\,du\,dy.
\end{eqnarray*}
Since the above random factors are independent, we have
\begin{eqnarray*}
&& \E \bigl[ \bigl(F_{n}(b)-F_{n}(a)
\bigr)^{2m} \bigr]
\\
&&\qquad\leq c_3 \int_{\R^{2md}} \int_{[a, b]\times[0,b-a]^{2m-1}
}
\prod_{i=1}^{m} \bigl(n^{1-Hd}
\bigl|f(y_{2i-1})f(y_{2i}) \bigr| \bigl(u^{-Hd}_{2i-1}u^{-Hd}_{2i}
\bigr)
\\
&&\hspace*{127.5pt}\qquad\quad{}\times\E \bigl( \bigl|e^{\iota{y_{2i-1}\cdot
X_{2i}}/({n^Hu^H_{2i}})}-1 \bigr| \\
&&\qquad\quad\hspace*{151pt}{}+ \bigl|e^{\iota{y_{2i-1}\cdot
X_{2i-1}}/({n^Hu^H_{2i-1}})}-1 \bigr| \bigr) \bigr)\,du\,dy.
\end{eqnarray*}
With the change of variables $x=y_{2i-1}$, $y=y_{2i}$, $u=u_{2i}$ and
$v=u_{2i-1}$, the above inequality can be rewritten as
%
\begin{eqnarray}\label{e311}
\hspace*{-4pt}&& \E \bigl[ \bigl(F_{n}(b)-F_{n}(a)
\bigr)^{2m} \bigr]
\nonumber\\
\hspace*{-4pt}&&\qquad\leq c_3 \int_{\R^{2d}} \int^b_a
\int_0^{b-a} n^{1-Hd} \bigl|f(x)f(y) \bigr|
(uv)^{-Hd}
\nonumber
\\
\hspace*{-4pt}&&\hspace*{71.3pt}\qquad\quad{} \times\E \bigl( \bigl|e^{\iota{y\cdot
X_2}/({n^Hu^H})}-1 \bigr|\nonumber\\[-8pt]\\[-8pt]
&&\qquad\quad\hspace*{95pt}{} + \bigl|e^{\iota{y\cdot
X_1}/({n^Hv^H})}-1 \bigr| \bigr) \,du\,dv\,dx\,dy
\nonumber
\\
\hspace*{-4pt}&&\hspace*{71.3pt}\qquad\quad{} \times\biggl( 2 \int_{\R^{2d}} \int_0^{b-a}
\int_0^{b-a} n^{1-Hd} \bigl|f(x)f(y) \bigr|
(uv)^{-Hd}
\nonumber
\\
\hspace*{-4pt}&&\hspace*{128pt}\qquad\quad{} \times\E \bigl(\bigl|e^{\iota{y\cdot
X_1}/({n^Hu^H})}-1\bigr| \bigr) \,du \,dv \,dx \,dy \biggr)^{m-1}.
\nonumber
\end{eqnarray}
Notice that $\int^b_a u^{-Hd} \,du\leq c_4(b-a)^{1-Hd}$. Then, by Lemma
\ref{lema11}, we obtain
%
\begin{eqnarray}\label{e312}\quad
&&\int_{\R^{2d}} \int^b_a
\int_0^{b-a} n^{1-Hd} \bigl|f(x)f(y) \bigr|
(uv)^{-Hd}
\nonumber\\
&&\hspace*{50pt}\quad{} \times\E \bigl( \bigl|e^{\iota{y\cdot
X_2}/({n^Hu^H})}-1 \bigr| + \bigl|e^{\iota{y\cdot
X_1}/({n^Hv^H})}-1 \bigr| \bigr) \,du\,dv\,dx\,dy
\\
&&\qquad\leq c_5 (b-a)^{1-Hd}\int_{\R^{2d}}
\bigl|f(x)f(y) \bigr| |y|^{
{1}/{H}-d} \,dx\,dy
\nonumber
\end{eqnarray}
and
%
\begin{eqnarray}
\label{e313} &&\int_{\R^{2d}} \int^{b-a}_0
\int_0^{b-a} n^{1-Hd} \bigl|f(x)f(y) \bigr|
(uv)^{-Hd}
\nonumber\\
&&\hspace*{61pt}\quad{} \times\E \bigl( \bigl|e^{\iota{y\cdot
X_1}/({n^Hu^H})}-1 \bigr| \bigr) \,du\,dv\,dx\,dy
\\
&&\qquad\leq c_6 (b-a)^{1-Hd}\int_{\R^{2d}}
\bigl|f(x)f(y) \bigr| |y|^{
{1}/{H}-d} \,dx\,dy.
\nonumber
\end{eqnarray}
Now Proposition \ref{prop1} follows from
(\ref{e311}), (\ref{e312}) and (\ref{e313}).
\end{pf}

Now we prove that the moments of $F_n(t)$ converge to the
corresponding moments of $W(L_t(0))$.

Fix a finite number of disjoint intervals $(a_i, b_i]$ with
$i=1,\ldots,N$ and
$b_i\le a_{i+1}$. Let $\mathbf{m}=(m_1, \ldots, m_N)$ be a fixed
multi-index with $m_i\in\N$ for $i=1,\ldots,N$.
Set $\sum_{i=1}^N m_i=|\mathbf{m}|$ and $\prod_{i=1}^N m_i!=\mathbf{m}!$.
We need to consider the following sequence of random variables:
\[
G_n=\prod_{i=1}^N \bigl(
F_n(b_i)- F_n(a_i)
\bigr)^{m_i}
\]
and compute $\lim_{n\rightarrow\infty} \E(G_n) $. Notice
that the expectation of $G_n$ can be written as
\[
\E(G_n) = \mathbf{m}! n^{{|\mathbf{m}|(1+Hd)}/2} \E \Biggl( \int
_{D_\mathbf{m}} \prod_{i=1}^N
\prod_{j=1}^{m_i}f\bigl(n^{H}B
\bigl(s^i_{j}\bigr)\bigr)\,ds \Biggr),
\]
where
%
\begin{equation}\label{e314}
D_\mathbf{m}= \bigl\{s \in\R^{|\mathbf{m}|}\dvtx  a_{i}<s^i_{1}<
\cdots <s^i_{m_i}<b_{i}, 1\le i\le N \bigr\}.
\end{equation}
Here and in the sequel we denote the coordinates of a point $s\in\R
^{|\mathbf{m}|}$ as
$s= (s^i_j)$, where $ 1\le i \le N$ and $1\le j \le m_i $.

For simplicity of notation, we define
\[
J_0= \bigl\{(i,j)\dvtx  1\leq i\leq N, 1\leq j\leq m_i
\bigr\}.
\]
For any $(i_1, j_1)$ and $(i_2,j_2)\in J_0$, we define the following
dictionary ordering:
\[
(i_1, j_1)\leq(i_2,j_2),
\]
if $i_1<i_2$ or $i_1=i_2$ and $j_1\leq j_2$. For any $(i,j)$ in $J_0$,
under the above ordering, $(i,j)$ is the $(\sum^{i-1}_{k=1}m_k+j)$th
element in $J_0$, and we define
$\#(i,j)=\sum^{i-1}_{k=1}m_k+j$.

\begin{proposition} \label{prop2} Suppose that at least one of the
exponents $m_i$ is odd. Then
\[
\lim_{n\to\infty}\E(G_n)=0.
\]
\end{proposition}
\begin{pf} The proof will be done in several steps.

\textit{Step} 1.
Using a similar argument to that in (\ref{i32}), we obtain
\begin{eqnarray*}
\E(G_n)
&=& \frac{\mathbf{m}!}{(2\pi)^{|\mathbf{m}|d}} n^{{|\mathbf
{m}|(1-Hd)}/{2}}\\[-2pt]
&&{}\times \int_{\R^{|\mathbf{m}|d}} \int
_{D_\mathbf{m}} \int_{\R^{|\mathbf{m}|d}} \Biggl( \prod
^{N}_{i=1}\prod^{m_i}_{j=1}
f\bigl(y^i_j\bigr) \Biggr)
\\[-2pt]
&&\hspace*{82.2pt}{} \times\exp \Biggl(-\frac{1}{2}\Var \Biggl( \sum
^{N}_{i=1}\sum^{m_i}_{j=1}
\xi^i_j\cdot B\bigl(s^i_j\bigr)
\Biggr) \\[-2pt]
&&\qquad\quad\hspace*{127pt}{}- \iota\sum^{N}_{i=1}\sum
^{m_i}_{j=1} \frac{y^i_j\cdot\xi^i_j}{n^H} \Biggr) \,d\xi\,ds\,dy
\\[-6pt]
&=& \frac{\mathbf{m}!}{(2\pi)^{|\mathbf{m}|d}} n^{{|\mathbf
{m}|(1-Hd)}/{2}} \\[-2pt]
&&\hspace*{0pt}{}\times\int_{\R^{|\mathbf{m}|d}} \int
_{D_\mathbf{m}} \int_{\R^{|\mathbf{m}|d}} \Biggl( \prod
^{N}_{i=1}\prod^{m_i}_{j=1}
f\bigl(y^i_j\bigr) \Biggr)
\\[-2pt]
&&\hspace*{84pt}{} \times\exp \Biggl(-\frac{1}{2}\Var \Biggl( \sum
^{N}_{i=1}\sum^{m_i}_{j=1}
\xi^i_j\cdot B\bigl(s^i_j\bigr)
\Biggr) \Biggr) \\[-2pt]
&&\hspace*{84pt}{} \times\prod^{N}_{i=1}\prod
^{m_i}_{j=1} \bigl(e^{ -\iota{y^i_j\cdot
\xi^i_j}/{n^H}}-1 \bigr)
\,d\xi\,ds\,dy,
\end{eqnarray*}
where we used the fact $\int_{\R^d} f(x)\,dx=0$ in the last
equality.\vadjust{\goodbreak}

By the local nondeterminism property (\ref{modified-lndp}), with the
convention $s^i_0=s^{i-1}_{m_{i-1}}$
for $2\leq i \leq N$ and $s^1_0=0$,
\[
\Var \Biggl( \sum^{N}_{i=1}\sum
^{m_i}_{j=1} \xi^i_j \cdot B
\bigl(s^i_j\bigr) \Biggr) \geq\kappa_H \sum
^{N}_{i=1}\sum
^{m_i}_{j=1} \biggl|\sum_{(l,k)\geq
(i,j)}
\xi^{l}_k \biggr|^2\bigl(s^i_j-s^i_{j-1}
\bigr)^{2H}.
\]

Let $F(y)=\prod^{N}_{i=1}\prod^{m_i}_{j=1} f(y^i_j)$, and make the
change of variables
\[
\eta^i_j=\sum_{(\ell,k)\geq(i,j)}\xi^{\ell}_k
\]
for $1\leq i \leq N$ and $1\leq j\leq m_i$. Then we can estimate
$E(G_n)$ as follows:
\begin{eqnarray*}
\bigl|\E(G_n)\bigr|
&\leq& c_1 n^{{|\mathbf{m}|(1-Hd)}/{2}} \\[-2pt]
&&{}\times\int_{\R^{|\mathbf{m}|d}} \int
_{D_\mathbf{m}} \int_{\R^{|\mathbf{m}|d}} \bigl|F(y)\bigr|\exp \Biggl(-
\frac{\kappa_H}{2}\sum^{N}_{i=1}\sum
^{m_i}_{j=1} \bigl|\eta^i_j\bigr|^2
\bigl(s^i_j-s^i_{j-1}
\bigr)^{2H} \Biggr)
\\[-2pt]
&&\hspace*{64.2pt}{} \times \Biggl( \prod^{N}_{i=1}\prod
^{m_i}_{j=1} \biggl|\exp \biggl( - \iota
\frac{y^i_j}{n^H}\cdot\bigl(\eta^i_j-
\eta^i_{j+1}\bigr) \biggr)-1 \biggr| \Biggr)\,d\eta\,ds\,dy.
\end{eqnarray*}
Making the change of variable $\eta^i_j=(s^i_j -s^i_{j-1})^H \xi^i_j$ yields
\begin{eqnarray*}
\hspace*{-4pt}&&\bigl|\E(G_n)\bigr|\\[-2pt]
\hspace*{-4pt}&&\qquad\leq c_1 n^{{|\mathbf{m}|(1-Hd)}/{2}} \\[-2pt]
\hspace*{-4pt}&&\qquad\quad{}\times\int_{\R^{|\mathbf{m}|d}} \int
_{D_\mathbf{m}} \int_{\R^{|\mathbf
{m}|d}} \bigl|F(y)\bigr|
\\[-2pt]
\hspace*{-4pt}&&\qquad\quad\hspace*{84pt}{} \times \Biggl(\prod^{N}_{i=1}\prod
^{m_i}_{j=1} \bigl(s^i_j-s^i_{j-1}
\bigr)^{-Hd} \Biggr) \exp \Biggl(-\frac{\kappa_H}{2}\sum
^{N}_{i=1}\sum^{m_i}_{j=1}
\bigl|\xi ^i_j\bigr|^2 \Biggr)
\\[-2pt]
\hspace*{-4pt}&&\hspace*{84pt}\qquad\quad{} \times \Biggl( \prod^{N}_{i=1}\prod
^{m_i}_{j=1} \biggl|\exp \biggl(\iota
\frac
{y^i_j\cdot\xi^i_{j+1}}{n^H
(s^i_{j+1}-s^i_{j})^H} \\[-2pt]
\hspace*{-4pt}&&\hspace*{158.4pt}\qquad\quad{}- \iota\frac{y^i_j \cdot\xi^i_j}{n^H
(s^i_{j}-s^i_{j-1})^H} \biggr)-1 \biggr| \Biggr)\,d\xi\,ds\,dy
\end{eqnarray*}
with the convention $\xi^N_{m_N+1}=0$, $\xi^{i}_{m_i+1}=\xi
^{i+1}_{1}$ for $1\leq i\leq N-1$. For the probabilistic argument to be
used below, it is convenient to express the integral with respect to
$d\xi$ as an expectation. In this way we can write
%
\begin{eqnarray}\label{e315}
&& \bigl|\E(G_n)\bigr|
\nonumber\\
&&\qquad \leq c_2 n^{{|\mathbf{m}|(1-Hd)}/{2}}\nonumber\\
&&\qquad\quad\hspace*{0pt}{}\times \int_{\R^{|\mathbf{m}|d}} \int
_{D_\mathbf{m}} \bigl|F(y)\bigr| \Biggl(\prod^{N}_{i=1}
\prod^{m_i}_{j=1} \bigl(s^i_j-s^i_{j-1}
\bigr)^{-Hd} \Biggr)
\\
&&\hspace*{56.8pt}\qquad\quad{} \times\E \Biggl( \prod^{N}_{i=1}\prod
^{m_i}_{j=1} \biggl|\exp \biggl(\iota
\frac{y^i_j\cdot X^i_{j+1}}{n^H
(s^i_{j+1}-s^i_{j})^H} \nonumber\\
&&\hspace*{171pt}{}- \iota\frac{y^i_j \cdot X^i_j}{n^H
(s^i_{j}-s^i_{j-1})^H} \biggr)-1 \biggr| \Biggr)\,ds\,dy,
\nonumber
\end{eqnarray}
where $\sqrt{\kappa_H}X^i_j$ ($1\leq i\leq N$, $1\leq j\leq m_i$) are
independent copies of a $d$-dimensional standard normal random vector,
and as before, we use the convention $X^N_{m_N+1}=0$,
$X^{i}_{m_i+1}=X^{i+1}_{1}$ for $1\leq i\leq N-1$.

Denote the expectation in (\ref{e315}) by $I$. That is,
\begin{eqnarray*}
I &=& \E \Biggl( \prod^{N}_{i=1}\prod
^{m_i}_{j=1} \biggl|\exp \biggl(\iota
\frac{y^i_j\cdot X^i_{j+1}}{n^H
(s^i_{j+1}-s^i_{j})^H} - \iota\frac{y^i_j \cdot X^i_j}{n^H
(s^i_{j}-s^i_{j-1})^H} \biggr)-1 \biggr| \Biggr)
\\
&=& \E \biggl( \prod_{(i,j)\in J_0} I_{i,j}
\biggr),
\end{eqnarray*}
where
\[
I_{i,j}= \biggl|\exp \biggl(\iota\frac{y^i_j\cdot X^i_{j+1}}{n^H
(s^i_{j+1}-s^i_{j})^H} - \iota
\frac{y^i_j \cdot X^i_j}{n^H
(s^i_{j}-s^i_{j-1})^H} \biggr)-1 \biggr|
\]
for $1\leq i\leq N$ and $1\leq j\leq m_i$.

Notice that the random variables $I_{i,j}$ for $(i,j)\in J_0$ are
dependent. We are going to choose a proper subset of $J_0$ in the
following way. Assume that $m_{\ell}$ is the first odd exponent. Then
we choose all the factors $I_{i,j}$ such that
$ \#(i,j)<\#(\ell,m_{\ell})$ and $\#(i,j)$ is odd. Then we choose all
the factors $I_{i,j}$ such that
$ \#(i,j)>\#(\ell,m_{\ell})+1$, and $\#(i,j)$ is even. Notice that
all these factors are mutually independent, and they are also
independent of the product $I_{\ell,m_{\ell}} I_{\ell+1,1}$. The
lack of independence of the two factors $I_{\ell,m_{\ell}} $ and
$I_{\ell+1,1}$ will be compensated by the fact that the integral of
$(s_1^{\ell+1} -s^\ell_{m_\ell})^{-\beta}$ is finite for any $\beta
<2$ because we have the constraint $s^\ell_{m_\ell}< b_\ell<
s_1^{\ell+1}$.
To make this argument more precise, let us define
\[
J_{\ell}=J_{\ell,1}\cup J_{\ell,2},
\]
where
\[
J_{\ell,1}=\bigl\{(i,j)\in J_0\dvtx  \#(i,j)<\#(
\ell,m_{\ell}) \mbox{ and } \#(i,j) \mbox{ odd}\bigr\}
\]
and
\[
J_{\ell,2}=\bigl\{(i,j)\in J_0\dvtx  \#(i,j)>\#(
\ell,m_{\ell})+1 \mbox{ and } \#(i,j) \mbox{ even}\bigr\}.
\]

Notice that $I_{i,j}\leq2$ for all $(i,j)\in J_0$. Then
\[
I \leq \cases{\displaystyle  c_2 \E \biggl(I_{\ell,m_{\ell}} I_{\ell+1,1}
\prod_{(i,j)\in J_{\ell}} I_{i,j} \biggr), &\quad if $\ell\neq
N$,
\vspace*{2pt}\cr
\displaystyle c_2 \E \biggl(I_{\ell,m_{\ell}} \prod
_{(i,j)\in J_{\ell}} I_{i,j} \biggr), &\quad if $\ell=N$.}
\]

\textit{Step} 2. We first consider the case $\ell\neq N$. In
this case, the number of elements in $J_\ell$ is $[\frac{|\mathbf
{m}|}{2}]-1$ and
\begin{eqnarray*}
\bigl|\E(G_n)\bigr| & \leq & c_3 n^{{|\mathbf{m}|(1-Hd)}/{2}} \int
_{\R^{|\mathbf{m}|d}} \int_{D_\mathbf{m}} \bigl|F(y)\bigr| \biggl(\prod
_{(i,j)\in J_0} \bigl(s^i_j-s^i_{j-1}
\bigr)^{-Hd} \biggr)
\\
&&\hspace*{110.5pt}{} \times\E (I_{\ell, m_{\ell}} I_{\ell+1,1} )\prod
_{(i,j)\in J_{\ell}} \E(I_{i, j})\,ds\,dy.
\end{eqnarray*}
In the last inequality, we used the fact that all random variables
$I_{\ell, m_{\ell}} I_{\ell+1,1}$ and $I_{i, j}$ for $(i,j)\in
J_{\ell}$ are independent.

Since $|e^{\iota(z_1-z_2)}-1|\leq|e^{\iota z_1}-1|+|e^{\iota z_2}-1|$
for all $z_1,z_2\in\R$,
\begin{eqnarray*}
\E (I_{\ell, m_{\ell}} I_{\ell+1,1} )&\leq& \E \bigl\{ \bigl(\bigl|e^{\iota{y^\ell_{m_\ell}\cdot
X^{\ell+1}_1}/({n^H (s^{\ell+1}_1-s^\ell_{m_\ell})^H})}
-1\bigr|\\
&&\hspace*{14.4pt}{}+\bigl|e^{ \iota
{y^\ell_{m_\ell} \cdot X^\ell_{m_\ell}}/({n^H
(s^\ell_{m_\ell}-s^\ell_{m_{\ell}-1})^H})}-1\bigr| \bigr)
\\
&&\hspace*{9pt}{} \times \bigl(\bigl|e^{\iota{y^{\ell+1}_1\cdot X^{\ell+1}_{2}}/({n^H(s^{\ell
+1}_{2}-s^{\ell+1}_{1})^H})}-1\bigr|\\
&&\hspace*{58.3pt}\hspace*{-34.4pt}{}+\bigl|e^{\iota{y^{\ell+1}_{1} \cdot
X^{\ell+1}_{1}}/({n^H
(s^{\ell+1}_{1}-s^\ell_{m_\ell})^H})}-1\bigr| \bigr) \bigr\}.
\end{eqnarray*}
Notice that $X^{\ell+1}_{2}, X^{\ell+1}_1$ and $X^\ell_{m_\ell}$
are independent. As a consequence, we can write
\[
\E (I_{\ell, m_{\ell}} I_{\ell+1,1} )\leq I^{\ell
}_1+I^{\ell}_2,
\]
where
\begin{eqnarray*}
I^{\ell}_1 &=& \E\bigl|e^{\iota{y^\ell_{m_\ell}\cdot X^{\ell
+1}_1}/({n^H (s^{\ell+1}_1-s^\ell_{m_\ell})^H})} -1\bigr| \E\bigl|e^{\iota
{y^{\ell+1}_1\cdot X^{\ell+1}_{2}}/({n^H(s^{\ell+1}_{2}-s^{\ell
+1}_{1})^H})}-1\bigr|
\\
&&{} + \E\bigl|e^{ \iota{y^\ell_{m_\ell} \cdot X^\ell_{m_\ell}}/({n^H
(s^\ell_{m_\ell}-s^\ell_{m_{\ell}-1})^H})}-1\bigr| \E\bigl|e^{\iota
{y^{\ell+1}_1\cdot X^{\ell+1}_{2}}/({n^H(s^{\ell+1}_{2}-s^{\ell
+1}_{1})^H})}-1\bigr|
\\
&&{} + \E\bigl|e^{ \iota{y^\ell_{m_\ell} \cdot X^\ell_{m_\ell}}/({n^H
(s^\ell_{m_\ell}-s^\ell_{m_{\ell}-1})^H})}-1\bigr| \E\bigl|e^{\iota
{y^{\ell+1}_{1} \cdot X^{\ell+1}_{1}}/({n^H
(s^{\ell+1}_{1}-s^\ell_{m_\ell})^H})}-1\bigr|
\end{eqnarray*}
and
%
\begin{equation}\label{e320a}
I^\ell_2=\E \bigl(\bigl|e^{\iota{y^\ell_{m_\ell}\cdot X^{\ell
+1}_1}/({n^H (s^{\ell+1}_1-s^\ell_{m_\ell})^H})} -1\bigr|\bigl|e^{\iota
{y^{\ell+1}_{1} \cdot X^{\ell+1}_{1}}/({n^H
(s^{\ell+1}_{1}-s^\ell_{m_\ell})^H})}-1\bigr|
\bigr).\hspace*{-35pt}
\end{equation}

Therefore,
%
\begin{equation}
\label{e321a} \bigl|\E(G_n)\bigr| \leq c_3 n^{{|\mathbf{m}|(1-Hd)}/{2}}\int
_{\R^{|\mathbf{m}|d}} \bigl|F(y)\bigr| (G_{1,n}+G_{2,n})\,dy,
\end{equation}
where
\[
G_{k,n}= \int_{D_\mathbf{m}} \biggl(\prod
_{(i,j)\in J_0} \bigl(s^i_j-s^i_{j-1}
\bigr)^{-Hd} \biggr) I^{\ell}_k \prod
_{(i,j)\in J_{\ell}} \E(I_{i, j})\,ds
\]
for $k=1,2$.

We claim that
\[
G_{1,n} \leq c_4 n^{([{|\mathbf{m}|}/{2}]+1)(Hd-1)} \bigl|y^{\ell
+1}_{1}\bigr|^{{1}/{H}-d}\bigl|y^\ell_{m_\ell}\bigr|^{{1}/{H}-d}
\prod_{(i,j)\in J_{\ell}} \bigl|y^i_j\bigr|^{{1}/{H}-d}.
\]
In fact, making the change of variables $v^i_j=s^i_j-s^i_{j-1}$ for all
$(i,j)\in J_0$, and defining
\[
a_{i,j}= \cases{\displaystyle  \bigl(v^i_j
v^i_{j+1}\bigr)^{-Hd} \E(I_{i, j}), &\quad if
$(i,j)\in J_{\ell} \mbox{ and } (i,j)\neq (N,m_N)$;
\vspace*{2pt}\cr
\displaystyle \bigl(v^N_{m_N}\bigr)^{-Hd} \E(I_{N, m_N}),
&\quad if $(i,j)\in J_{\ell} \mbox{ and } (i,j)=(N,m_N)$,}
\]
and $a^{\ell}_1=(v^{\ell+1}_{2}v^{\ell+1}_{1}v^{\ell}_{m_\ell
})^{-Hd}  I^\ell_1$, we obtain
%
\begin{eqnarray}
\label{e322a} G_{1,n} &\leq& \int_{[0,b_N]^{|\mathbf{m}|}}
a^{\ell}_1 \prod_{(i,j)\in
J_{\ell}}
a_{i, j} \,dv
\nonumber\\
&=& \int_{[0,b_N]^3} a^{\ell}_1
\,dv^{\ell+1}_{2} \,dv^{\ell+1}_{1}
\,dv^{\ell
}_{m_\ell} \prod_{(i,j)\in J_{\ell}} \int
_{[0,b_N]^2} a_{i, j} \,dv^i_j
\,dv^i_{j+1}
\\
&\leq& c_5 n^{([{|\mathbf{m}|}/{2}]+1)(Hd-1)}
\bigl|y^{\ell+1}_{1}\bigr|^{{1}/{H}-d}\bigl|y^\ell_{m_\ell}\bigr|^{{1}/{H}-d}
\prod_{(i,j)\in J_{\ell}} \bigl|y^i_j\bigr|^{{1}/{H}-d}.
\nonumber
\end{eqnarray}
Here we used Lemma \ref{lema11} in the last inequality $[\frac
{|\mathbf{m}|}{2}]+1$ times.

For any $\beta\in[0,1]$, we have $|e^{\iota z}-1|\leq c_{\beta}
|z|^{\beta}$ for all $z\in\R$. Recall the definition of $I^\ell_2$
in (\ref{e320a}). We then have
\begin{eqnarray*}
I^\ell_2 &\leq& c_6 n^{-2H\beta}\bigl|y^\ell_{m_\ell}\bigr|^{\beta}\bigl|y^{\ell
+1}_{1}\bigr|^{\beta}
\bigl(s^{\ell+1}_1-s^\ell_{m_\ell}
\bigr)^{-2H\beta} \E \bigl|X^{\ell+1}_1\bigr|^{2\beta}
\\
&\leq& c_7 n^{-2H\beta}\bigl|y^\ell_{m_\ell}\bigr|^{\beta}\bigl|y^{\ell
+1}_{1}\bigr|^{\beta}
\bigl(s^{\ell+1}_1-s^\ell_{m_\ell}
\bigr)^{-2H\beta}.
\end{eqnarray*}
So
\begin{eqnarray*}
G_{2, n} &\leq& c_7 n^{-2H\beta}\bigl|y^\ell_{m_\ell}\bigr|^{\beta}\bigl|y^{\ell
+1}_{1}\bigr|^{\beta}
\\
&&{} \times\int_{D_\mathbf{m}} \bigl(s^{\ell+1}_1-s^\ell_{m_\ell
}
\bigr)^{-2H\beta} \biggl(\prod_{(i,j)\in J_0}
\bigl(s^i_j-s^i_{j-1}
\bigr)^{-Hd} \biggr) I^{\ell
}_2 \prod
_{(i,j)\in J_{\ell}} \E(I_{i, j}) \,ds.
\end{eqnarray*}

Define $J_{\ell,3}=\{(i,j)\in J_0\dvtx  \#(i,j)<\#(\ell,m_\ell)\}$ and
\[
D^\ell_\mathbf{m}= \bigl\{ a_{i}<s^i_{1}<
\cdots<s^i_{m_i}<b_{i}, 1\le i\le\ell-1;
a_{\ell}<s^{\ell}_{1}<\cdots<s^\ell_{m_\ell-1}<b_{\ell}
\bigr\}.
\]

Integrating the above integral with respect to $s^i_j$ for $(i,j)\in
J_{\ell,2}$ and using Lemma \ref{lema11},
\begin{eqnarray*}
G_{2,n} &\leq& c_8 n^{-2H\beta} n^{\# J_{\ell, 2}(Hd-1)}
\bigl|y^\ell _{m_\ell}\bigr|^{\beta}\bigl|y^{\ell+1}_{1}\bigr|^{\beta}
\prod_{(i,j)\in J_{\ell,2}}\bigl|y^i_j\bigr|^{{1}/{H}-d}
\\
&&{} \times\int_{D^\ell_\mathbf{m}} I^{\ell}_3 \prod
_{(i,j)\in J_{\ell,3}} \bigl(s^i_j-s^i_{j-1}
\bigr)^{-Hd} \prod_{(i,j)\in
J_{\ell,1}}
\E(I_{i, j})\,ds,
\end{eqnarray*}
where $\# J_{\ell, 2}$ is the cardinality of $J_{\ell, 2}$ and
\begin{eqnarray*}
I^{\ell}_3 &=& \int^{b_{l}}_{s^{\ell}_{m_\ell-1}}
\int^{b_{l+1}}_{a_{\ell+1}} \int^{b_{l+1}}_{s^{\ell+1}_1}
\bigl(s^{\ell
+1}_{2}-s^{\ell+1}_{1}
\bigr)^{-Hd} \bigl(s^{\ell+1}_1-s^\ell_{m_\ell
}
\bigr)^{-Hd-2H\beta}
\\
&&\hspace*{79pt}{} \times\bigl(s^\ell_{m_\ell}-s^\ell_{m_{\ell}-1}
\bigr)^{-Hd} \,ds^{\ell+1}_{2} \,ds^{\ell+1}_1
\,ds^\ell_{m_\ell}.
\end{eqnarray*}
We observe that if $1-Hd<2H\beta\leq2-2Hd$,
\begin{eqnarray*}
I^{\ell}_3 & \leq & c_9 \int
^{b_{l}}_{s^{\ell}_{m_\ell-1}} \int^{b_{l+1}}_{a_{\ell+1}}
\bigl(s^{\ell+1}_1-s^\ell_{m_\ell}
\bigr)^{-Hd-2H\beta
} \bigl(s^\ell_{m_\ell}-s^\ell_{m_{\ell}-1}
\bigr)^{-Hd} \,ds^{\ell+1}_1 \,ds^\ell_{m_\ell}
\\
& \leq & c_{10} \int^{b_{l}}_{s^{\ell}_{m_\ell-1}}
\bigl(a_{\ell+1}-s^\ell _{m_\ell}\bigr)^{1-Hd-2H\beta}
\bigl(s^\ell_{m_\ell}-s^\ell_{m_{\ell
}-1}
\bigr)^{-Hd} \,ds^\ell_{m_\ell}
\\
& \leq & c_{11} (a_{\ell+1}-a_\ell)^{2-2Hd-2H\beta}.
\end{eqnarray*}
As a consequence,
%
\begin{eqnarray}\label{e323}
G_{2,n} &\leq& c_{12} n^{-2H\beta} n^{\# J_{\ell, 2}(Hd-1)}
\bigl|y^\ell _{m_\ell}\bigr|^{\beta}\bigl|y^{\ell+1}_{1}\bigr|^{\beta}
\prod_{(i,j)\in J_{\ell,2}}\bigl|y^i_j\bigr|^{{1}/{H}-d}
\nonumber\\
&&{} \times\int_{D^{\ell}_\mathbf{m}} \prod_{(i,j)\in J_{\ell,3}}
\bigl(s^i_j-s^i_{j-1}
\bigr)^{-Hd} \prod_{(i,j)\in
J_{\ell,1}}
\E(I_{i, j})\,ds
\nonumber\\[-8pt]\\[-8pt]
&\leq& c_{13} n^{-2H\beta} n^{\# J_{\ell}(Hd-1)} \bigl|y^\ell
_{m_\ell}\bigr|^{\beta}\bigl|y^{\ell+1}_{1}\bigr|^{\beta}
\prod_{(i,j)\in J_{\ell
}}\bigl|y^i_j\bigr|^{{1}/{H}-d}
\nonumber
\\
& = & c_{13} n^{-2H\beta} n^{([{|\mathbf{m}|}/{2}]-1)(Hd-1)} \bigl|y^\ell_{m_\ell}\bigr|^{\beta}\bigl|y^{\ell+1}_{1}\bigr|^{\beta}
\prod_{(i,j)\in
J_{\ell}}\bigl|y^i_j\bigr|^{{1}/{H}-d}.
\nonumber
\end{eqnarray}

Choosing $\beta=\frac{1}{H}-d$ in (\ref{e323}),
%
\begin{eqnarray}\label{e324}\qquad
G_{2,n} &\leq& c_{13} n^{([{|\mathbf{m}|}/{2}]+1)(Hd-1)}
\bigl|y^\ell_{m_\ell}\bigr|^{{1}/{H}-d}\bigl|y^{\ell+1}_{1}\bigr|^{{1}/{H}-d}\nonumber\\[-8pt]\\[-8pt]
&&\hspace*{0pt}\times
\prod_{(i,j)\in J_{\ell}}\bigl|y^i_j\bigr|^{{1}/{H}-d}.\nonumber
\end{eqnarray}

Substituting (\ref{e322a}) and (\ref{e324}) into (\ref
{e321a}) yields
%
\begin{eqnarray}
\label{e325}\qquad
\bigl|\E(G_n)\bigr|
&\leq& c_{14} n^{-({1-Hd})/{2}}\nonumber\\
&&{}\times\int_{\R^{|\mathbf{m}|d}} \bigl|F(y)\bigr|
\bigl|y^\ell_{m_\ell}\bigr|^{{1}/{H}-d}
\bigl|y^{\ell+1}_{1}\bigr|^{{1}/{H}-d}\\
&&\hspace*{36.5pt}{}\times \prod_{(i,j)\in
J_{\ell}}\bigl|y^i_j\bigr|^{{1}/{H}-d} \,dy. \nonumber
\end{eqnarray}

\textit{Step} 3. Now we consider the case $\ell=N$. In this
case, $J_{\ell}=J_{\ell,1}$ and
\begin{eqnarray*}
\bigl|\E(G_n)\bigr| & \leq & c_{15} n^{{|\mathbf{m}|(1-Hd)}/{2}} \\
&&{}\times\int
_{\R^{|\mathbf{m}|d}} \int_{D_\mathbf{m}} \bigl|F(y)\bigr| \prod
_{(i,j)\in J_0} \bigl(s^i_j-s^i_{j-1}
\bigr)^{-Hd}
\\
&&\hspace*{56.7pt}{} \times\E(I_{N, m_N} )\prod_{(i,j)\in J_{N,1}}
\E(I_{i,
j}) \,ds \,dy.
\end{eqnarray*}
Define $J_{N,3}=\{(i,j)\in J_0\dvtx  \#(i,j)<\#(N,m_N)\}$ and
\[
D^N_\mathbf{m}= \bigl\{ a_{i}<s^i_{1}<
\cdots<s^i_{m_i}<b_{i}, 1\le i\le N-1;
a_{\ell
}<s^{N}_{1}<\cdots<s^N_{m_N-1}<b_{N}
\bigr\}.
\]
Integrating the above integral with respect to $s^N_{m_N}$ and using
Lemma \ref{lema11} yield
\begin{eqnarray*}
\bigl|\E(G_n)\bigr| & \leq & c_{16} n^{{(|\mathbf{m}|-2)(1-Hd)}/{2}}\\
&&{}\times \int
_{\R^{|\mathbf{m}|d}} \int_{D^N_\mathbf{m}} \bigl|F(y)\bigr|
\bigl|y^N_{m_N}\bigr|^{{1}/{H}-d}
\\
&&\hspace*{57.2pt}{} \times\prod_{(i,j)\in J_{N,3}} \bigl(s^i_j-s^i_{j-1}
\bigr)^{-Hd} \prod_{(i,j)\in J_{N,1}}
\E(I_{i, j}) \,ds \,dy.
\end{eqnarray*}
Using arguments similar to those in step 2,
%
\begin{eqnarray}
\label{e326} \bigl|\E(G_n)\bigr| &\leq& c_{17}
n^{-({1-Hd})/{2}}\nonumber\\[-8pt]\\[-8pt]
&&{}\times
\int_{\R^{|\mathbf{m}|d}} \bigl|F(y)\bigr| \bigl|y^N_{m_N}\bigr|^{{1}/{H}-d}
\prod_{(i,j)\in J_{N,1}} \bigl|y^i_j\bigr|^{{1}/{H}-d}
\,dy.\nonumber
\end{eqnarray}

\textit{Step} 4. Recall that $f\in H^{{1}/{H}-d}_0$. Then
from (\ref{e325}) and (\ref{e326}), we see that $|\E(G_n)|$ is
bounded by a multiple of $n^{-({1-Hd})/{2}}$. Our result now follows
from taking the limit.
\end{pf}

In the sequel, we consider the convergence of moments when all
exponents $m_i$ are even.
Recall the definition of $D_\mathbf{m}$ in (\ref{e314}). For
$1\leq\ell\leq N$ and $1\leq k\leq\frac{m_{\ell}}{2}$, we define
\[
O^\ell_k=D_\mathbf{m}\cap \biggl\{
\frac{s^\ell_{2k}-s^\ell
_{2k-2}}{2}<s^\ell_{2k}-s^\ell_{2k-1}
\biggr\}.
\]
The following result tells us that the integrals over the domain
$O^\ell_k$ do not contribute to the limit of the moments. This result
will play a fundamental role in computing the limits of even moments.

\begin{proposition} \label{prop3}
For any $1\leq\ell\leq N$ and $1\le k\le\frac{m_{\ell}}{2}$,
\[
\lim_{n\to\infty} n^{{|\mathbf{m}|(1+Hd)}/{2}} \E \Biggl(\int
_{O^\ell_k}\prod^N_{i=1}
\prod_{j=1}^{m_i} f\bigl(n^{H}B
\bigl(s^i_j\bigr)\bigr) \,ds \Biggr)=0.
\]
\end{proposition}

\begin{pf} Using the arguments and notation in the proof of Proposition
\ref{prop2}, we obtain
\begin{eqnarray*}
&&n^{{|\mathbf{m}|(1+Hd)}/{2}} \Biggl\llvert \E \Biggl(\int_{O^\ell
_k}\prod
^N_{i=1}\prod
_{j=1}^{m_i} f\bigl(n^{H}B
\bigl(s^i_j\bigr)\bigr) \,ds \Biggr)\Biggr\rrvert
\\
&&\qquad \leq c_1 n^{{|\mathbf{m}|(1-Hd)}/{2}}\\
&&\qquad\quad{}\times \int_{\R^{|\mathbf{m}|d}} \int
_{O^\ell_k} \bigl|F(y)\bigr| \Biggl(\prod^{N}_{i=1}
\prod^{m_i}_{j=1} \bigl(s^i_j-s^i_{j-1}
\bigr)^{-H} \Biggr)
\\
&&\hspace*{53.7pt}\qquad\quad{} \times\E \Biggl( \prod^{N}_{i=1}\prod
^{m_i}_{j=1} \biggl|\exp \biggl(\iota
\frac{y^i_j\cdot X^i_{j+1}}{n^H
(s^i_{j+1}-s^i_{j})^H} \\
&&\qquad\quad\hspace*{136pt}{}- \iota\frac{y^i_j \cdot X^i_j}{n^H
(s^i_{j}-s^i_{j-1})^H} \biggr)-1 \biggr| \Biggr) \,ds \,dy,
\end{eqnarray*}
where $X^N_{m_N+1}=0$, $X^{i}_{m_i+1}=X^{i+1}_{1}$ for $1\leq i\leq
N-1$, and $\sqrt{\kappa_H}X^i_j$ ($1\leq i\leq N$, $1\leq j\leq m_i$)
are independent copies of a $d$-dimensional standard normal random vector.

We make the change of variables $v^{i}_j=s^i_{j}-s^i_{j-1}$ for all
$(i,j)\in J_0$. The integral domain $O^\ell_k$ becomes
\[
D^\ell_k= \biggl\{v\in\R^{|\mathbf{m}|}_+\dvtx
a_1<\sum_{(i,j)\in J_0} v^{i}_j
< b_N, v^\ell_{2k-1}<v^\ell_{2k}
\biggr\}.
\]

For $(i,j)\in J_0$, define
\[
I_{i,j}= \biggl|\exp \biggl(\iota\frac{y^i_j\cdot X^i_{j+1}}{n^H
(v^i_{j+1})^H} - \iota
\frac{y^i_j \cdot X^i_j}{n^H
(v^i_{j})^H} \biggr)-1 \biggr|.
\]
Then
%
\begin{eqnarray}\label{e331}
&&
n^{{|\mathbf{m}|(1+Hd)}/{2}} \biggl|\E \Biggl(\int_{O^\ell
_k}\prod
^N_{i=1}\prod_{j=1}^{m_i}
f\bigl(n^{H}B\bigl(s^i_j\bigr)\bigr) \,ds
\Biggr) \biggr|
\nonumber\\[-8pt]\\[-8pt]
&&\qquad \leq c_1 n^{{|\mathbf{m}|(1-Hd)}/{2}} \int_{\R^{|\mathbf{m}|d}} \int
_{D^\ell_k} F(y) \E \biggl( \prod_{(i,j)\in J_0}
\bigl(v^i_j \bigr)^{-Hd}I_{i,j}
\biggr) \,dv \,dy.
\nonumber
\end{eqnarray}
Next we estimate the expectation in (\ref{e331}). We are going to
use an argument similar to the one used in the proof of Proposition
\ref{prop2}, based on the selection of some factors in the above product. Here,
the dependent product that will play a basic role will be $I_{\ell,
2k} I_{\ell,2k-1}$, due to the definition of the set $O_k^\ell$.
Define
\[
J^{\ell}_{k}=J^{\ell}_{k,1}\cup
J^{\ell}_{k,2},
\]
where
\begin{eqnarray*}
J^{\ell}_{k,1}&=&\bigl\{(i,j)\in J_0\dvtx  \#(i,j)<\#(
\ell,2k-2), \#(i,j) \mbox{ odd} \bigr\},
\\
J^{\ell}_{k,2}&=&\bigl\{(i,j)\in J_0\dvtx  \#(i,j)>\#(
\ell,2k), \#(i,j) \mbox{ even} \bigr\}.
\end{eqnarray*}

Since all exponents $m_i$ are even, the number of elements in $J^{\ell
}_k$ is $\frac{|\mathbf{m}|-2}{2}=\frac{|\mathbf{m}|}{2}-1$. From
the definition of $I_{i,j}$, we know that random variables $I_{\ell,
2k} I_{\ell,2k-1}$ and $I_{i,j}$ for $(i,j)\in J^{\ell}_{k}$ are
independent. Then
\begin{eqnarray*}
&&
n^{{|\mathbf{m}|(1+Hd)}/{2}} \Biggl|\E \Biggl(\int_{O^\ell
_k}\prod
^N_{i=1}\prod_{j=1}^{m_i}
f\bigl(n^{H}B\bigl(s^i_j\bigr)\bigr) \,ds
\Biggr) \Biggr|
\\
&&\qquad \leq c_2 n^{{|\mathbf{m}|(1-Hd)}/{2}} \int_{\R^{|\mathbf{m}|d}} \int
_{D^\ell_k} F(y) \E(I_{\ell,
2k}I_{\ell,2k-1})
\\
&&\hspace*{108.7pt}\qquad\quad{} \times\prod_{(i,j)\in J_0} \bigl(v^i_j
\bigr)^{-Hd} \prod_{(i,j)\in J^{\ell}_{k}}
\E(I_{i,j}) \,dv \,dy.
\end{eqnarray*}

For $(i,j)\in J^{\ell}_k$ and $(i,j)\neq(N,m_N)$, define
\[
a_{i,j}=\bigl(v^i_j v^i_{j+1}
\bigr)^{-Hd} \E(I_{i, j})
\]
and $a_{N,m_N}=(v^N_{m_N})^{-Hd}  \E(I_{N, m_N})$. From Lemma \ref
{lema11}, we obtain
\[
\int_{[0,b_N]^2} a_{i,j} \,d v^i_{j}
\,d v^i_{j+1}\leq c_3 n^{Hd-1}
\bigl|y^i_j\bigr|^{{1}/{H}-d}
\]
for all $(i,j)\in J^{\ell}_k$. Therefore,
%
\begin{eqnarray}\label{e332}
&&
n^{{|\mathbf{m}|(1+Hd)}/{2}} \Biggl|\E \Biggl(\int_{O^\ell
_k}\prod
^N_{i=1}\prod_{j=1}^{m_i}
f\bigl(n^{H}B\bigl(s^i_j\bigr)\bigr) \,ds
\Biggr) \Biggr|
\nonumber\\[-8pt]\\[-8pt]
&&\qquad \leq c_4 n^{1-Hd} \int_{\R^{|\mathbf{m}|d}} F(y)
\prod_{(i,j)\in J^{\ell}_{k}} \bigl|y^i_j\bigr|^{{1}/{H}-d}
I^{\ell}_k \,dy,
\nonumber
\end{eqnarray}
where
\[
I^{\ell}_k=\int^{b_N}_0
\int^{b_N}_{v^{\ell}_{2k-1}}\int^{b_N}_{0}
\bigl(v^{\ell}_{2k+1}v^{\ell}_{2k}v^{\ell}_{2k-1}
\bigr)^{-Hd} \E(I_{\ell,
2k}I_{\ell,2k-1}) \,dv^{\ell}_{2k+1}
\,dv^{\ell}_{2k} \,dv^{\ell}_{2k-1}.
\]

Notice that $|e^{\iota(z_1-z_2)}-1|\leq|e^{\iota z_1}-1|+|e^{\iota
z_2}-1|$ for all $z_1, z_2\in\R$. Using the independence of $X^\ell
_{2k-1}, X^\ell_{2k}$ and $X^\ell_{2k+1}$, we obtain
\[
\E (I_{\ell, 2k-1} I_{\ell,2k} )\leq A^{\ell}_{k,1}+A^{\ell}_{k,2},
\]
where
\begin{eqnarray*}
A^{\ell}_{k,1}&=& \E\bigl|e^{\iota{y^\ell_{2k-1}\cdot X^\ell
_{2k}}/({n^H (v^\ell_{2k})^H})} -1\bigr|
\E\bigl|e^{\iota{y^\ell
_{2k}\cdot X^\ell_{2k+1}}/({n^H(v^\ell_{2k+1})^H})}-1\bigr|
\\
&&{} +\E\bigl|e^{ \iota{y^\ell_{2k-1} \cdot X^\ell
_{2k-1}}/({n^H(v^\ell_{2k-1})^H})}-1\bigr| \E\bigl|e^{\iota{y^\ell
_{2k}\cdot X^\ell_{2k+1}}/({n^H(v^\ell_{2k+1})^H})}-1\bigr|
\\
&&{} +\E\bigl|e^{ \iota{y^\ell_{2k-1} \cdot X^\ell
_{2k-1}}/({n^H(v^\ell_{2k-1})^H})}-1\bigr| \E\bigl|e^{\iota{y^\ell_{2k}
\cdot X^\ell_{2k}}/({n^H
(v^\ell_{2k})^H})}-1\bigr|
\end{eqnarray*}
and
\[
A^{\ell}_{k,2}=\E \bigl( \bigl|e^{\iota{y^\ell_{2k-1}\cdot X^\ell
_{2k}}/({n^H (v^\ell_{2k})^H})} -1\bigr|
\bigl|e^{\iota{y^\ell_{2k} \cdot
X^\ell_{2k}}/({n^H
(v^\ell_{2k})^H})}-1\bigr| \bigr).
\]

Now we have
%
\begin{equation}\label{e333}
I^{\ell}_k=I^{\ell}_{k,1}+I^{\ell}_{k,2},
\end{equation}
where
\[
I^{\ell}_{k,i}=\int^{b_N}_0
\int^{b_N}_{v^{\ell}_{2k-1}}\int^{b_N}_{0}
\bigl(v^{\ell}_{2k+1}v^{\ell}_{2k}v^{\ell}_{2k-1}
\bigr)^{-Hd} A^{\ell}_{k,i} \,dv^{\ell}_{2k+1}
\,dv^{\ell}_{2k} \,dv^{\ell}_{2k-1}
\]
for $i=1,2$. By Lemma \ref{lema11},
%
\begin{equation}\label{e334}
I^{\ell}_{k,1}\leq c_5 n^{-2(1-Hd)}\bigl|y^\ell_{2k-1}\bigr|^{
{1}/{H}-d}\bigl|y^\ell_{2k}\bigr|^{{1}/{H}-d}.
\end{equation}
For any $\beta\in[0,1]$, we have $|e^{\iota z}-1|\leq c_{\beta
}|z|^{\beta}$ for all $z\in\R$. Then
\[
A^{\ell}_{k,2}\leq c_6 n^{-2H\beta}
\bigl(v^\ell_{2k}\bigr)^{-2H\beta} \bigl|y^\ell_{2k-1}\bigr|^{\beta}\bigl|y^\ell_{2k}\bigr|^{\beta}.
\]
Therefore, if $1-Hd<2H\beta< 2-2Hd$,
\begin{eqnarray*}
I^{\ell}_{k,2} &\leq& c_6 n^{-2H\beta}
\bigl|y^\ell_{2k-1}\bigr|^{\beta}\bigl|y^\ell
_{2k}\bigr|^{\beta} \\
&&\hspace*{0pt}{}\times\int^{b_N}_0\int^{b_N}_{v^{\ell}_{2k-1}}\int^{b_N}_{0}
\bigl(v^{\ell}_{2k+1}\bigr)^{-Hd}
\bigl(v^{\ell}_{2k}\bigr)^{-Hd-2H\beta}\\
&&\hspace*{81pt}{}\times \bigl(v^{\ell}_{2k-1}
\bigr)^{-Hd} \,dv^{\ell}_{2k+1} \,dv^{\ell}_{2k}
\,dv^{\ell}_{2k-1}
\\
&\leq& c_7 n^{-2H\beta} \bigl|y^\ell_{2k-1}\bigr|^{\beta}\bigl|y^\ell
_{2k}\bigr|^{\beta} \int^{b_N}_0
\bigl(v^{\ell}_{2k-1}\bigr)^{1-2Hd-2H\beta} \,dv^{\ell}_{2k-1}
\\
&\leq& c_8 n^{-2H\beta} \bigl|y^\ell_{2k-1}\bigr|^{\beta}\bigl|y^\ell
_{2k}\bigr|^{\beta}.
\end{eqnarray*}

Choose $\beta=\frac{3(1-Hd)}{4H}$,
%
\begin{equation}\label{e335}
I^{\ell}_{k,2} \leq c_8 n^{-{3(1-Hd)}/{2}}
\bigl|y^\ell _{2k-1}\bigr|^{{3(1-Hd)}/({4H})}\bigl|y^\ell_{2k}\bigr|^{{3(1-Hd)}/({4H})}.
\end{equation}

Substituting (\ref{e334}) and (\ref{e335}) into (\ref
{e333}), we obtain
%
\begin{eqnarray}\label{e336}
I^{\ell}_k&\leq& c_9 n^{-{3(1-Hd)}/{2}}
\bigl(\bigl|y^\ell _{2k-1}\bigr|^{{3(1-Hd)}/({4H})}\bigl|y^\ell_{2k}\bigr|^{
{3(1-Hd)}/({4H})}
\nonumber\\[-8pt]\\[-8pt]
&&\hspace*{112.6pt}{}+\bigl|y^\ell_{2k-1}\bigr|^{{1}/{H}-d}\bigl|y^\ell_{2k}\bigr|^{
{1}/{H}-d}
\bigr).\nonumber
\end{eqnarray}

Our result now follows easily from (\ref{e332}), (\ref{e336})
and the assumption $f\in H^{{1}/{H}-d}_0$.
\end{pf}

Consider now the convergence of moments when all exponents $m_i$ are even.
On each portion of the coordinates $a_{i}<s^i_{1}<\cdots
<s^i_{m_i}<b_{i}$ we make the following change of variables:
\[
u^i_{2k}= s^i_{2k} \quad\mbox{and}\quad
u^i_{2k-1}=n\bigl( s^i_{2k}-
s^i_{2k-1}\bigr)\qquad \mbox{where } 1\le k \le m_i/2
\]
with the convention $u^i_0=s^i_0=a_i$. The idea is to couple each
variable with an odd
subindex with the next one. In this way we obtain
%
\begin{eqnarray}
\label{e341}
\E(G_n) &=& \mathbf{m}! n^K \E \Biggl( \int_{D^n_{\mathbf{m}}} \prod
_{i=1}^N \prod
_{k=1}^{{m_i}/2}  f\bigl(n^{H}B
\bigl(u^i_{2k}\bigr)\bigr)\nonumber\\[-8pt]\\[-8pt]
&&\hspace*{92.4pt}{}\times f\biggl(n^{H}B
\biggl(u^i_{2k}- \frac{u^i_{2k-1}}n \biggr)\biggr) \,du
\Biggr),
\nonumber
\end{eqnarray}
where $K$ and $D_{{\mathbf m}}^n$ are as follows:
\[
K = \frac{|\mathbf{m}|Hd} 2
\]
and
\begin{eqnarray*}
D^n_\mathbf{m} &=& \biggl\{ u\in\R^{|\mathbf{m}|}\dvtx
a_i<u^i_2 <u^i_4<
\cdots <u^i_{m_i} <b_i;
\\
&&\hspace*{6.5pt} 0< u^i_{2k-1} <n\bigl(u^i_{2k}
-u^i_{2k-2}\bigr), 1\le k\le\frac
{m_i}2 \biggr\}.
\end{eqnarray*}

We compute the expectation (\ref{e341}) in the following way. Define
the $\mathbf{m}$-dimensional Gaussian random vector $X(u)$ by
\[
X^i_{2k} (u)= B\bigl(u^i_{2k}
\bigr) \quad\mbox{and}\quad X^i_{2k-1}(u)=n^H \biggl(B
\biggl(u^i_{2k}- \frac{u^i_{2k-1}}n \biggr) - B
\bigl(u^i_{2k}\bigr) \biggr),
\]
where $1\le k \le\frac{m_i}2$. The covariance matrix and
the probability density function of the Gaussian random vector $X(u)$
are denoted by $Q_{n}(u)$ and
\[
p_{n}(x) = (2\pi)^{- {|\mathbf{m}|d}/2}\bigl(\det Q_{n}(u)
\bigr)^{-{1}/{2}}\exp \bigl(-\tfrac{1}{2} xQ_{n}(u)^{-1}x^{T}
\bigr),
\]
respectively. With the above notation we can write
\[
\E(G_n) = \mathbf{m}! n^K \int_{\R^{|\mathbf{m}|d}}
\int_{D_\mathbf{m}} \prod_{i=1}^N
\prod_{k=1}^{{m_i}/2} f\bigl(n^{H}x^i_{2k}
\bigr)f\bigl(n^{H} x^i_{2k}
+x^i_{2k-1}\bigr) p_n(x) \,du \,dx.
\]

Making the change of variables $y^i_j=n^Hx^i_j$ if $j$ is even, and
$y^i_j= x^i_j$ if $j$ is odd, we then obtain
%
\begin{equation}
\label{eq10} \E(G_n) = \mathbf{m}! \int_{\R^{|\mathbf{m}|d}}
\int_{D^n_\mathbf{m}} \prod_{i=1}^N
\prod_{k=1}^{{m_i}/2} f\bigl(
y^i_{2k} \bigr)f\bigl( y^i_{2k}
+y^i_{2k-1}\bigr) p_n\bigl(y(n)\bigr) \,du \,dy,\hspace*{-35pt}
\end{equation}
where $y ^i_j(n)= n^{-H} y^i_j$ if $j$ is even and $y
^i_j(n)= y^i_j$ if $j$ is odd.

\begin{proposition} \label{prop4} Suppose that all exponents $m_i$ are
even. Then
%
\begin{equation}
\lim_{n\to\infty}\E(G_n)=C_{H,d}^{{|\mathbf{m}|}/{2}}
\Vert f\Vert^{ |\mathbf{m}| }_{{1}/{H}-d } \mathbb{E} \Biggl( \prod
_{i=1}^N \bigl( W\bigl(L_{b_i}(0)\bigr)-
W\bigl(L_{a_i}(0)\bigr) \bigr)^{m_i} \Biggr).\hspace*{-35pt}
\end{equation}
\end{proposition}

\begin{pf}
Notice that we can find a sequence of functions $f_N$, which are
infinitely differentiable with compact support, such that
$\int_{\R^d} f_N(x)  \,dx=0$ and
\[
\lim_{N\to\infty} \int_{\R^d}
\bigl|f(x)-f_N(x)\bigr| \bigl(|x|^{1/H
-d} \vee1 \bigr) \,dx=0.
\]
So, by Proposition \ref{prop1}, we can assume that
$f$ is infinitely differentiable with compact support and
$\int_{\R^d} f (x)  \,dx=0$.

The equation (\ref{eq10}) can be written as
%
\begin{equation}
\label{eq10a} \E(G_n) = \mathbf{m}! \int_{\R^{|\mathbf{m}|d}}
\int_{D^n_\mathbf{m}} F(y) p_n\bigl(y(n)\bigr) \,du \,dy,\vadjust{\goodbreak}
\end{equation}
where
\[
F(y)=\prod_{i=1}^N\prod
_{k=1}^{ {m_i}/2} f\bigl(y^i_{2k}
\bigr) f\bigl(y^i_{2k} +y^i_{2k-1}
\bigr).
\]
The proof will be done in several steps.

\textit{Step} 1.
Let us compute the limit of the density $p_n(y(n))$ as $n$ tends to infinity.
We split the random vector $X(u)$ into two random vectors $X(u)=(Y(u),Z_n(u))$,
where $Y(u)$ contains the components of $X(u)$ with even subindices,
and $Z_n (u)$ contains the components with odd subindices.
That is, $Y(u)$ is an $\frac{|\mathbf{m}|d} 2$-dimensional random
vector, such that
$Y^i_k(u)=B(u^i_{2k})$ for $ 1 \le i\le N$ and $1\le k\le
\frac{m_i}2$. We denote by $A (u)$ the covariance matrix of $Y(u)$,
which does not depend on $n$. On the other hand, the covariance
matrix between the components of $Z_n (u)$ and $Y(u)$ converges to
the zero matrix, and the covariance matrix of the random vector
$Z_n(u)$ converges to
a diagonal matrix with entries equal to $(u^i_{2k-1})^{2H}$, $1 \le
i\le N$, and $1\le k\le\frac{m_i}2$. Therefore,
\begin{eqnarray*}
\lim_{n\rightarrow\infty} p_{n}\bigl(y(n)\bigr)
&=& (2\pi)^{-
{|\mathbf{m}|d}/2} \bigl(\det A (u)\bigr)^{-{1}/{2}}\\
&&{}\times \prod
_{i=1}^N \prod_{k=1}^{ {m_i}/2}
\bigl(u^i_{2k-1}\bigr)^{-Hd} \exp \biggl( -
\frac{|y^i_{2k-1}|^2} { 2(u^i_{2k-1})^{2H}} \biggr).
\end{eqnarray*}
On the other hand, the region $D_\mathbf{m}^{ n}$ converges, as $n$
tends to infinity, to
\[
\biggl\{ u\in\R^{ |\mathbf{m}| }\dvtx  a_i<u^i_2 <
\cdots<u^i_{m_i} <b_i; 0<u^i_{2k-1}<
\infty; 1\le k\le\frac{m_i} 2, 1\le i\le N \biggr\}.
\]
Notice that we can add a term $-1$ because $\int_{\R^d}
F(y)  \,dy^i_{2k-1} =0$ for any $i,k$, and
\begin{eqnarray*}
&&\int_0^\infty u^{-Hd} \bigl[
e^{ -{ |y^i_{2k-1}|^2}/
({2u^{2H}})} -1 \bigr] \,du\\
&&\qquad =-\bigl|y^i_{2k-1}\bigr|^{1/H-d}
\int_0^\infty u^{-Hd} \bigl[ 1-
e^{ -{1}/({2u^{2H}})} \bigr] \,du.
\end{eqnarray*}
Therefore, provided that we can interchange the limit with the
integrals in the expression (\ref{eq10a}), we obtain
%
\begin{eqnarray}
\label{equ1}
\lim_{n\rightarrow\infty} \E(G_n)
&=& \mathbf{m}! 2^{-{|\mathbf{m}|}/2} (2\pi)^{-
{|\mathbf{m}|d}/4} C_{H,d}^{ {|\mathbf{m}|}/2 }
\| f\|_{1/H-d}^{ | \mathbf{m}| } \nonumber\\[-8pt]\\[-8pt]
&&{}\times\int_{O_{ {\mathbf{m}}/2} } \bigl(\det
A (w)\bigr)^{-{1}/{2}} \,dw,
\nonumber
\end{eqnarray}
where
\[
O_{ {\mathbf{m}}/2} = \bigl\{ w\in\R^{ { |\mathbf{m}| }/2 }\dvtx  a_i<w^i_1
< \cdots<w^i_{{m_i}/2} <b_i, 1\le i\le N \bigr\},
\]
and $A(w)=A(u)$ with the change of variable $w_k^i=u^i_{2k}$.
Finally, the right-hand side of (\ref{equ1}) can also be written as
%
\begin{eqnarray}
\label{equ3} &&\Biggl(\prod_{i=1}^N
\frac{m_i! }{2^{ {m_i}/2} ({m_i}/2)
!}C_{H,d}^{{m_i}/2} \| f\|_{1/H-d}^{ m_i }
\Biggr) \nonumber\\[-8pt]\\[-8pt]
&&\qquad{}\times\int_{\prod_{i=1}^N
[a_i,b_i]^{{m_i}/2}} (2\pi)^{- {|\mathbf{m}|d}/4} \bigl(\det A (w)
\bigr)^{-{1}/{2}} \,dw,\nonumber
\end{eqnarray}
and, taking into account Lemma \ref{lema2}, this would finish the proof.

\textit{Step} 2.
In order to justify the passage of the
limit inside the integrals, we decompose the region
$D^n_{\mathbf{m}}$ into two components as follows. For $K>0$, we
define
\[
D^n_{\mathbf{m},K,1} = \biggl\{ u\in D^n_{\mathbf{m}}\dvtx
0< u^i_{2k-1} <K\wedge n\bigl(u^i_{2k}
-u^i_{2k-2}\bigr); 1\le k\le\frac{m_i}2 \biggr\}
\]
and $D^n_{\mathbf{m},K,2} = D^n_{\mathbf{m}}-D^n_{\mathbf{m},K,1}
$. Then, $\E(G_n) =I^1_{n,K} +I^2_{n,K} $, where
\begin{eqnarray*}
I^1_{n,K} & = & \mathbf{m}! \int_{\R^{|\mathbf{m}|d}}
\int_{D^n_{\mathbf{m},K,1}} F(y) p_n\bigl(y(n)\bigr) \,du \,dy,
\\
I^2_{n,K} & = & \mathbf{m}! \int_{\R^{|\mathbf{m}|d}}
\int_{D^n_{\mathbf{m},K,2}} F(y) p_n\bigl(y(n)\bigr) \,du
\,dy.
\end{eqnarray*}
The region $ D^n_{\mathbf{m},K,1} $ is uniformly
bounded in $n$, and we can then interchange the limit and the
integral with respect to $u$, provided that we have a uniform
integrability condition. To do this we need the following estimate
of the density $p_n(y(n))$.

For any $\xi\in\mathbb{R}^{|\mathbf{m}|}$ with components
$(\xi^i_j)$, $1\le i \le N$, $1\le j\le m_i$, we can write
\begin{eqnarray*}
\langle\xi,X \rangle &=& \sum_{i=1}^N \Biggl( \sum
_{j=1}^{{m_i}/2} \xi_{2j}^i \cdot
B\bigl(u^i_{2j}\bigr) \\[-0.5pt]
&&\hspace*{20.6pt}{}+\sum_{k=1}^{{m_i}/2}
\xi_{2k-1}^i \cdot n^H \biggl(B
\biggl(u^i_{2k } -\frac{u^i_{2k-1}}n\biggr) - B
\bigl(u^i_{2k-2}\bigr) \biggr) \Biggr)
\\[-0.5pt]
&=& \sum_{i=1}^N \sum
_{k=1}^{{m_i}/2} \biggl( \sum
_{(\ell, 2j)
\ge(i,2k) } \xi^\ell_{2j} \biggr) \cdot
\biggl( B\biggl(u^i_{2k } -\frac
{u^i_{2k-1}}n\biggr) -B
\bigl(u^i_{2k-2}\bigr)\biggr)
\\[-0.5pt]
&&{} + \sum_{i=1}^N \sum
_{k=1}^{{m_i}/2} \biggl( \sum
_{(\ell, 2j) \ge(i,2k) } \xi^\ell_{2j} -n^H
\xi_{2k-1}^i \biggr) \cdot \biggl(B\bigl(u^i_{2k}
\bigr)- B\biggl(u^i_{2k } -\frac{u^i_{2k-1}}n\biggr)\!
\biggr).
\end{eqnarray*}
Here we have used the ordering $(\ell, 2j) \ge(i, 2k)$ if $\ell>i$
or $\ell=i$ and $j\ge k$.\vadjust{\goodbreak}

By the local nondeterminism property (\ref{modified-lndp}),
%
\begin{eqnarray}
\label{equ2} \Var \langle\xi,X \rangle
&\geq& k_{H} \Biggl[ \sum_{i=1}^N
\sum_{k=1}^{{m_i}/2} \biggl| \sum
_{(\ell, 2j) \ge(i,2k) } \xi^\ell_{2j} \biggr| ^2
\cdot \biggl( u^i_{2k } -\frac{u^i_{2k-1}}n -
u^i_{2k-2} \biggr)^{2H}
\nonumber
\\[-0.5pt]
&&\hspace*{30.5pt}{} + \sum_{i=1}^N \sum
_{k=1}^{{m_i}/2} \biggl| \sum_{(\ell, 2j) \ge(i,2k) }
\xi^\ell_{2j} -n^H \xi^i_{2k-1}
\biggr| ^2 \biggl( \frac{u^i_{2k-1}}n \biggr)^{2H} \Biggr]
\nonumber\\[-0.5pt]
&=& k_{H} \Biggl[ \sum_{i=1}^N
\sum_{k=1}^{{m_i}/2} \bigl| \eta^i_{2k}
\bigr|^2 \biggl( u^i_{2k } -\frac{u^i_{2k-1}}n -
u^i_{2k-2} \biggr)^{2H}
\\[-0.5pt]
&&\hspace*{22pt}{} + \sum_{i=1}^N \sum
_{k=1}^{{m_i}/2} \bigl| \eta^i_{2k}
-n^H \eta_{2k-1}^i \bigr|^2 \biggl(
\frac{u^i_{2k-1}}n \biggr)^{2H} \Biggr]
\nonumber
\\[-0.5pt]
&=:& k_{H} R(\eta),\nonumber
\end{eqnarray}
where we have made the change of variables
\[
\eta^i_{2k}=\sum_{(\ell, 2j) \ge(i,2k) }
\xi^\ell_{2j} \quad\mbox{and}\quad \eta^i_{2k-1}=
\xi^i_{2k-1}.
\]
This implies that
\begin{eqnarray*}
(\det Q_{n})^{-{1}/{2}} &=& (2\pi)^{-{|\mathbf{m}|d}/2 }\int
_{\R^{| \mathbf
{m}|d}}\exp \biggl(-\frac{1}{2}\Var \langle\xi,X \rangle
\biggr) \,d\xi
\\[-0.5pt]
&\leq& (2\pi)^{-{|\mathbf{m}|d}/2}\int_{\R^{{|\mathbf
{m}|d}}} \exp \biggl(-
\frac{k_{H}}{2}R(\eta) \biggr) \,d\eta
\\[-0.5pt]
&=& c_1 \prod_{i=1}^N \prod
_{k=1}^{{m_i}/2} \bigl( u^i_{2k-1}
\bigr)^{-Hd} \biggl( u^i_{2k } -\frac{u^i_{2k-1}}n
- u^i_{2k-2} \biggr)^{-Hd}.
\end{eqnarray*}
Therefore,
%
\begin{equation}\label{e1}
p_{n}\bigl(y(n)\bigr) \leq c_2 \prod
_{i=1}^N \prod_{k=1}^{{m_i}/2}
\bigl( u^i_{2k-1} \bigr)^{-Hd} \biggl(
u^i_{2k } -\frac{u^i_{2k-1}}n - u^i_{2k-2}
\biggr)^{-Hd}.
\end{equation}
As a consequence of (\ref{e1}) and the inequality (\ref{ineq5}) in
Lemma \ref{lema5},
\begin{eqnarray*}
&& \int_{D^n_{\mathbf{m},K,1}} p_n\bigl(y(n)\bigr) \,du
\\
&&\qquad \leq c_3\int_{D^n_{\mathbf{m},K,1}} \prod
_{i=1}^N \prod_{k=1}^{{m_i}/2}
\bigl( u^i_{2k-1} \bigr)^{-Hd} \biggl(
u^i_{2k
} -\frac{u^i_{2k-1}}n - u^i_{2k-2}
\biggr)^{-Hd}\,du
\\
&&\qquad \leq c_4,
\end{eqnarray*}
where $c_4$ is a constant independent of $n$ and $y$. Thus, taking into
account that the function $F(y)$ is integrable, by the dominated
convergence theorem we obtain
\[
\lim_{n\rightarrow\infty} I^1_{n,K} = \mathbf{m}! \int
_{\R
^{|\mathbf{m}|d}} F(y) \biggl( \lim_{n\rightarrow\infty} \int
_{D^n_{\mathbf{m},K,1}} p_n\bigl(y(n)\bigr) \,du \biggr) \,dy.
\]
On the other hand, again by (\ref{e1}) and Lemma \ref{lema5},
there exists $p>1$ such that
%
\begin{equation}
\sup_n \int_{D^n_{\mathbf{m},K,1}} \bigl|p_n
\bigl(y(n)\bigr)\bigr|^p \,du<\infty,
\end{equation}
which implies
\[
\lim_{n\rightarrow\infty} I^1_{n,K} = \mathbf{m}! \int
_{\R
^{|\mathbf{m}|d}} \int_{\R^{|\mathbf{m}|}} F(y) \lim
_{n\rightarrow
\infty} \mathbf{1}_{D^n_{\mathbf{m},K,1}} (u)p_n\bigl(y(n)
\bigr) \,du \,dy.
\]
With the same notation as above we get
%
\begin{eqnarray}\label{2314}\qquad
\lim_{n\rightarrow\infty} I^1_{n,K}
&=& \mathbf{m} ! (2\pi)^{-{|\mathbf{m}|d}/2 } \biggl(\int_{O_{ {\mathbf{m}}/2} }
\bigl(\det A (w)\bigr)^{-{1}/{2}} \,dw \biggr)
\nonumber\\
&&{} \times\int_{\R^{|\mathbf{m}|d}} \prod_{i=1}^N
\prod_{k=1}^{{m_i}/2} \biggl( f
\bigl(y^i_{2k}\bigr) f\bigl(y^i_{2k}
+ y^i_{2k-1}\bigr) \\
&&\hspace*{77pt}{}\times\int_0^K
u^{-Hd} \bigl( e^{-{|y^i_{2k-1}|^2}/({2u^{2H}})}-1 \bigr) \,du \biggr) \,dy.
\nonumber
\end{eqnarray}
The right-hand side of the above equality converges to the term in
(\ref{equ3}) as $K$ tends to
infinity.

\textit{Step} 3. Now it suffices to show that
%
\begin{equation}
\label{eq4} \lim_{K\to\infty}\limsup_{n\rightarrow\infty}
I^2_{n,K}=0.
\end{equation}
First we observe that
\[
D^n_{\mathbf{m},K,2} = \bigcup_{i=1}^N
\bigcup_{k=1}^{{m_i}/2} D^n_{\mathbf{m},K,i,k},
\]
where
\[
D^n_{\mathbf{m},K,i,k} = \bigl\{ u\in D^n_{\mathbf{m}}\dvtx
u^i_{2k-1}\geq K\wedge n\bigl(u^{i}_{2k}-u^{i}_{2k-2}
\bigr) \bigr\}.
\]
So we only need to show that
%
\begin{equation}
\label{eq5} \lim_{K\to\infty}\limsup_{n\rightarrow\infty} \int
_{\R
^{|\mathbf{m}|d}} \int_{ \bigcup_{i=1}^N \bigcup_{k=1}^{{m_i}/2}
D^n_{\mathbf{m},K,i,k}} F(y) p_n
\bigl(y(n)\bigr) \,du \,dy =0.
\end{equation}

As a consequence of Proposition \ref{prop3},
we can replace $D^n_{\mathbf{m},K,i,k}$ in (\ref{eq5}) with
\[
D^{n,1}_{\mathbf{m},K,i,k}= \biggl\{ u\in D^n_{\mathbf{m}};
K\leq u^i_{2k-1}\leq\frac{n(u^{i}_{2k}-u^{i}_{2k-2})}{2} \biggr\}
\]
and just show that
%
\begin{equation}
\label{eq6} \lim_{K\to\infty}\limsup_{n\rightarrow\infty} \int
_{\R
^{|\mathbf{m}|d}} \int_{ \bigcup_{i=1}^N \bigcup_{k=1}^{{m_i}/2}
D^{n,1}_{\mathbf{m},K,i,k}} F(y) p_n
\bigl(y(n)\bigr) \,du \,dy =0.
\end{equation}

To do this we need more refined estimates of the density
$p_n(y(n))$. By Fourier analysis
\begin{eqnarray*}
p_n\bigl(y(n)\bigr) &=& (2\pi)^{- |\mathbf{m}|d }\int
_{\R^{|\mathbf{m}|d}} \exp \Biggl( -\frac{1}{2}\Var \langle\xi,X
\rangle
\\
&&\hspace*{94.1pt}{} -\iota \sum_{i=1}^{N} \sum
_{j=1}^{{m_i}/2} \biggl(\frac{ \xi^i_{2j}
\cdot y^i_{2j} }{n^H} +
\xi^i_{2j-1}\cdot y^i _{2j-1} \biggr)
\Biggr) \,d\xi.
\end{eqnarray*}

We choose a set $J$ of indexes of the form $(i,2j-1)$, where $1\le i
\le N$ and $1\le j\le\frac{m_i} 2$. For each index in $J$ we
introduce the operator
\[
\Delta_{i,2j-1} F\bigl(y^i_{2j-1}\bigr)= F
\bigl(y^i_{2j-1}\bigr)-F(0)
\]
and set $\Delta_{J}=\prod_{(i,2j-1) \in J}\Delta_{i,2j-1}$. Taking
into account that the integral on the variable $y^i_{2j-1}$ is zero, we
can replace $p_n(y(n))$ in (\ref{eq6}) by $\Delta_{J}p_n(y(n))$.
Using (\ref{equ2}), we obtain the following estimate:
%
\begin{eqnarray}
&& \bigl|\Delta_{J}p_n\bigl(y(n)\bigr) \bigr|
\nonumber\\
&&\qquad\leq c_5 \int_{\R^{|\mathbf{m}|d}}\exp \Biggl( -
\frac{k_{H}}{2} \sum^N_{i=1}\sum
^{{m_i}/{2}}_{j=1} \biggl( \bigl| \eta^i_{2j}
\bigr| ^2 \biggl( u^i_{2j } -\frac{u^i_{2j-1}}n -
u^i_{2j-2} \biggr)^{2H}
\nonumber
\\
&&\hspace*{125.5pt}\qquad\quad{} + \bigl| \eta^i_{2j} -n^H \eta_{2j-1}^i
\bigr|^2 \biggl( \frac
{u^i_{2j-1}}n \biggr)^{2H} \biggr) \Biggr)\nonumber\\
&&\qquad\quad\hspace*{34.3pt}{}\times
\prod_{i=1}^N \prod
_{j=1}^{{m_i}/2} \bigl|e^{-\iota\eta
^i_{2j-1} \cdot y^i_{2j-1} } -1 \bigr| \,d\eta
\\
&&\qquad= c_5 n^{- {|\mathbf{m}|Hd}/{2}} \nonumber\\
&&\qquad\quad{}\times\int_{\R^{|\mathbf
{m}|d}}\exp
\Biggl( -\frac{k_{H}}{2} \sum^N_{i=1}
\sum^{{m_i}/{2}}_{j=1} \biggl( \bigl|
\eta^i_{2j} \bigr|^2 \biggl( u^i_{2j }
-\frac{u^i_{2j-1}}n - u^i_{2j-2} \biggr)^{2H}
\nonumber
\\
&&\hspace*{172.2pt}\qquad\quad{} + \bigl| \eta_{2j-1}^i \bigr|^2 \biggl(
\frac{u^i_{2j-1}}n \biggr)^{2H} \biggr) \Biggr) \nonumber\\
&&\hspace*{39pt}\qquad\quad\times\prod
_{i=1}^N \prod_{j=1}^{{m_i}/2}
\bigl| e^{- \iota{( \eta^i_{2j} - \eta_{2j-1}^i )\cdot y^i_{2j-1} }/
{ n^H}} -1 \bigr| \,d\eta.
\nonumber
\end{eqnarray}
This shows that
%
\begin{equation}
\label{eq7} \biggl| \int_{\R^{|\mathbf{m}|d}} \int_{ \bigcup_{i=1}^N \bigcup
_{k=1}^{{m_i}/2} D^{n,1}_{\mathbf{m},K,i,k}} F(y)
p_n\bigl(y(n)\bigr) \,du \,dy \biggr| \leq c_6 \sum
_{i=1}^N \sum_{k=1}^{{m_i}/2}
I_{\mathbf{m},K,i,k},\hspace*{-35pt}
\end{equation}
where
\begin{eqnarray*}
&&
I_{\mathbf{m},K,i,k}
\\
&&\qquad= n^{- {|\mathbf{m}|Hd}/{2}}
\\
&&\qquad\quad{} \times\int_{D^{n,1}_{\mathbf{m},K,i,k}} \int_{\R^{2|\mathbf{m}|d}}
\bigl|F(y)\bigr|\\
&&\qquad\quad\hspace*{70pt}{}\times\exp \Biggl( -\frac{k_{H}}{2} \sum^N_{i=1}
\sum^{{m_i}/{2}}_{j=1} \biggl( \bigl|
\eta^i_{2j} \bigr| ^2 \biggl( u^i_{2j }
-\frac{u^i_{2j-1}}n - u^i_{2j-2} \biggr)^{2H}
\\
&&\qquad\quad\hspace*{207pt}{} + \bigl| \eta_{2j-1}^i \bigr| ^2 \biggl(
\frac{u^i_{2j-1}}n \biggr)^{2H} \biggr) \Biggr) \\
&&\qquad\quad\hspace*{77.5pt}{}\times\prod
_{i=1}^N \prod_{j=1}^{{m_i}/2}
\bigl| e^{-\iota{( \eta^i_{2j} - \eta_{2j-1}^i ) \cdot
y^i_{2j-1} }/ { n^H}} -1 \bigr| \,d\eta \,dy \,du.
\end{eqnarray*}

To estimate $I_{\mathbf{m},K,i,k}$ on the right-hand side of (\ref
{eq7}), we first consider
the integral in the variables $u=u^i_{2k}$, $v=u^i_{2k-1}$,
$w=\eta^i_{2k}$ and $z=\eta^i_{2k-1}$. Set $u^i_{2k-2}=u_0$
and $y^i_{2k-1}=y$. That is, we have the integral
\begin{eqnarray*}
\hspace*{-4pt}&& I_1(u_0,b_i,y)
\\
\hspace*{-4pt}&&\qquad= n^{-Hd} \int_{u_{0}}^{b_i} \int
_{K}^{{n(u-u_0)}/{2}} \int_{\R^{2d}} \exp
\biggl(-\frac{\kappa_H}{2} \biggl( |z|^2 \biggl(\frac vn
\biggr)^{2H} \\
\hspace*{-4pt}&&\qquad\quad\hspace*{170.1pt}{}+|w|^2 \biggl(u -u_0-\frac vn
\biggr)^{2H} \biggr) \biggr)
\\
\hspace*{-4pt}&&\hspace*{100.4pt}\qquad\quad{} \times \bigl( \bigl|e^{-\iota{z\cdot y
}/{n^H}}-1 \bigr|+ \bigl|e^{-\iota{w\cdot y }/{n^H}}-1 \bigr| \bigr) \,dw \,dz \,dv
\,du.
\end{eqnarray*}
We see that
\[
I_1(u_0,b_i,y) = I_2(u_0,b_i,y)+I_3(u_0,b_i,y),
\]
where
\begin{eqnarray*}
&&
I_2(u_0,b_i,y)
\\
&&\qquad= n^{-Hd} \int_{u_{0}}^{b_i} \int
_{K}^{{n(u-u_0)}/{2}} \int_{\R^{2d}} \exp
\biggl(-\frac{\kappa_H}{2} \biggl( |z|^2 \biggl(\frac vn
\biggr)^{2H} \\
&&\qquad\quad\hspace*{170.6pt}{}+|w|^2 \biggl(u -u_0-\frac vn
\biggr)^{2H} \biggr) \biggr)
\\
&&\hspace*{114pt}\qquad\quad{} \times \bigl|e^{-\iota{z \cdot y }/{n^H}}-1 \bigr| \,dw \,dz \,dv \,du
\end{eqnarray*}
and
\begin{eqnarray*}
&&
I_3(u_0,b_i,y)
\\
&&\qquad= n^{-Hd} \int_{u_{0}}^{b_i} \int
_K^{{n(u-u_0)}/{2}} \int_{\R
^{2d}} \exp
\biggl(-\frac{\kappa_H}{2} \biggl( |z|^2 \biggl(\frac vn
\biggr)^{2H}\\
&&\hspace*{203.4pt}{} +|w|^2 \biggl(u -u_0-\frac vn
\biggr)^{2H} \biggr) \biggr)
\\
&&\hspace*{114pt}\qquad\quad{} \times \bigl|e^{-\iota{w \cdot y }/{n^H}}-1 \bigr| \,dw \,dz \,dv \,du.
\end{eqnarray*}
Integrating with respect to $w$ and using the inequality
(\ref{ineq4}) in the \hyperref[app]{Appendix} lead us to
%
\begin{eqnarray}
\label{e320}
&&I_2(u_0,b_i,y)\nonumber\\
&&\qquad = c
n^{-Hd} \int_{u_{0}}^{b_i} \int
_K^{{n(u-u_0)}/{2}} \int_{\R^d} \biggl(u
-u_0-\frac vn \biggr)^{-Hd}
\nonumber\\
&&\qquad\quad\hspace*{115.5pt}{} \times e^{ -({\kappa_H |z|^2 }/{2})
(v/n)^{2H} } \bigl|e^{-\iota{z \cdot y }/{n^H}}-1 \bigr| \,dz \,dv \,du
\\
&&\qquad \le c K^{1-Hd-H\beta} |y|^{\beta} \int_{u_{0}}^{b_i}
(u-u_0)^{-Hd} \,du
\nonumber
\\
&&\qquad \le c K^{1-Hd-H\beta} (b_i-u_0)^{1-Hd}
|y|^{\beta}.
\nonumber
\end{eqnarray}
In a similar way, but with the application of (\ref{ineq3}) instead of
(\ref{ineq4}), we obtain
%
\begin{eqnarray}
\label{e321} && I_3(u_0,b_i,y)
\nonumber\\
&&\qquad = c \int_{u_{0}}^{b_i} \int_K^{{n(u-u_0)}/{2}}
v^{-Hd} \int_{\R^d} e^{ -({\kappa_H |w|^2 }/{2})  (u -u_0-v/n
)^{2H} } \nonumber\\
&&\qquad\quad\hspace*{114pt}{}\times\bigl|
e^{-\iota{w\cdot y }/{n^H}}-1 \bigr| \,dw \,dv \,du
\\
&&\qquad \le c n^{1-Hd-H\beta} |y|^{\beta} \int_{u_{0}}^{b_i}
(u-u_0)^{1-2Hd-H\beta} \,du
\nonumber
\\
&&\qquad \le c n^{1-Hd-H\beta} (b_i-u_0)^{2-2Hd-H\beta}
|y|^{\beta}.
\nonumber
\end{eqnarray}
Combining (\ref{e320}) and (\ref{e321}) gives
%
\begin{eqnarray}
\label{e322} && I_1(u_0,b_i,y)
\nonumber\\[-2pt]
&&\qquad\le c K^{1-Hd-H\beta} (b_i-a_i)^{1-Hd}
|y|^{\beta}
\\[-2pt]
&&\qquad\quad{} + c n^{1-Hd-H\beta} (b_i-a_i)^{2-2Hd-H\beta}
|y|^{\beta}.
\nonumber
\end{eqnarray}
Once this is done, we proceed to consider the integrals in the
variables $u=u^j_{2l}$, $v=u^j_{2l-1}$, $z=\eta^j_{2l}$ and
$w=\eta^j_{2l-1}$ with indices $(j,l)\neq(i,k)$. Set also
$u^j_{2l-2}=u_0$ and $y^j_{2l-1}=y$. That is, we have the integral
\begin{eqnarray*}
&& I_4(u_0,b_j,y)
\\[-2pt]
&&\qquad= n^{-Hd} \int_{u_{0}}^{b_j} \int
_0^{n(u-u_0)} \int_{\R^{2d}} \exp
\biggl(-\frac{\kappa_H}{2} \biggl( |z|^2 \biggl(\frac vn
\biggr)^{2H} \\[-2pt]
&&\qquad\quad\hspace*{163.4pt}{}+|w|^2 \biggl(u -u_0-\frac vn
\biggr)^{2H} \biggr) \biggr)
\\[-2pt]
&&\hspace*{107.7pt}\qquad\quad{} \times \bigl( \bigl|e^{-\iota{z\cdot y
}/{n^H}}-1 \bigr|\\[-2pt]
&&
\qquad\quad\hspace*{124.7pt}{}
+ \bigl|e^{-\iota{w\cdot y }/{n^H}}-1 \bigr| \bigr) \,dw \,dz \,dv
\,du.
\end{eqnarray*}
We can decompose this integral into two components,
\[
I_4(u_0,b_j,y)=I_5(u_0,b_j,y)+I_6(u_0,b_j,y),
\]
where
\begin{eqnarray*}
&& I_5(u_0,b_j,y)
\\[-2pt]
&&\qquad= n^{-Hd} \int_{u_{0}}^{b_j} \int
_0^{n(u-u_0)} \int_{\R^{2d}} \exp
\biggl(-\frac{\kappa_H}{2} \biggl( |z|^2 \biggl(\frac vn
\biggr)^{2H} \\[-2pt]
&&\qquad\quad\hspace*{163.4pt}{}+|w|^2 \biggl(u -u_0-\frac vn
\biggr)^{2H} \biggr) \biggr)
\\[-2pt]
&&\hspace*{106.6pt}\qquad\quad{} \times \bigl|e^{-\iota{z\cdot y }/{n^H}}-1 \bigr| \,dw \,dz \,dv \,du
\\[-2pt]
&&\qquad= c n^{-Hd} \int_{u_{0}}^{b_j} \int
_0^{n(u-u_0)} \int_{\R^d} \biggl(u
-u_0-\frac vn \biggr)^{-Hd} e^{ -({\kappa_H |z|^2 }/{2})
(v/n)^{2H} }
\\[-2pt]
&&\hspace*{104.6pt}\qquad\quad{} \times \bigl|e^{-\iota{z\cdot y }/{n^H}}-1 \bigr| \,dz \,dv
\,du
\end{eqnarray*}
and
\begin{eqnarray*}
&& I_6(u_0,b_j,y)
\\[-2pt]
&&\qquad= n^{-Hd} \int_{u_{0}}^{b_j} \int
_0^{n(u-u_0)} \int_{\R^{2d}} \exp
\biggl(-\frac{\kappa_H}{2} \biggl( |z|^2 \biggl(\frac vn
\biggr)^{2H} \\[-2pt]
&&\qquad\quad\hspace*{163.4pt}{}+|w|^2 \biggl(u -u_0-\frac vn
\biggr)^{2H} \biggr) \biggr)
\\[-2pt]
&&\hspace*{107.5pt}\qquad\quad{} \times \bigl|e^{-\iota{w\cdot y }/{n^H}}-1 \bigr| \,dw \,dz \,dv \,du
\\[-2pt]
&&\qquad= c \int_{u_{0}}^{b_j} \int_0^{n(u-u_0)}
v^{-Hd} \int_{\R^d} e^{ -({\kappa_H |w|^2 }/{2})  (u -u_0-v/n  )^{2H}
} \\[-2pt]
&&\qquad\quad\hspace*{110pt}{}\times\bigl|
e^{-\iota{w\cdot y }/{n^H}}-1 \bigr| \,dw \,dv \,du.
\end{eqnarray*}
Inequalities (\ref{ineq2}) and (\ref{ineq1}) imply that
\begin{eqnarray*}
I _4(u_0,b_j,y) & \le & c \int
_{u_{0}}^{b_j} (u-u_0)^{-Hd}
|y|^{{1}/H-d} \,du
\\
& = & c (b_j-u_0)^{1-Hd} |y|^{{1}/H-d}
\\
& \le & c (b_j-a_j)^{1-Hd} |y|^{{1}/H-d}.
\end{eqnarray*}
The remaining integrals can be dealt with in the same way as
$I_4(u_0, b_j, y)$. Thus statement (\ref{eq6}) follows. The
proof is complete.
\end{pf}

\begin{pf*}{Proof of Theorem \ref{th1}} This follows
from Lemma \ref{lema2}, Propositions \ref{prop1}, \ref{prop2} and
\ref{prop4} by the method of moments.
\end{pf*}

\begin{Remark*}\label{Rem}
Although the constant $C_{H,d}$ is finite for $H>\frac1{d+2}$, the
proof of Theorem \ref{th1} only works for $H> \frac1{d+1}$.
The reason for this is that for any $y\in\mathbb{R}^d$,
\[
\int_{\mathbb{R}^d} \exp \biggl( -\frac\kappa2 |\xi|^2
u^{2H} \biggr) \bigl(1- e^{\iota\xi\cdot y} \bigr) \,d\xi = \biggl(
\frac{2\pi} \kappa \biggr)^{d/2} u^{-Hd} \biggl(1- \exp
\biggl( - \frac{ |y|^2}{2\kappa u^{2H}} \biggr) \biggr),
\]
which is bounded by $c|y|^2u^{-H(d+2)}$, while, on the other hand,
\[
\int_{\mathbb{R}^d} \exp \biggl( -\frac\kappa2 |\xi|^2
u^{2H} \biggr) \bigl\llvert e^{\iota\xi\cdot y} -1 \bigr\rrvert \,d\xi \le
c |y| u^{-H(d+1)}.
\]
So, any type of estimation procedure, like the one based on the local
nondeterminism property used in this paper, will lead to an upper bound
of the form $u^{-H(d+1)}$.
\end{Remark*}

\begin{appendix}\label{app}
\section*{Appendix}

Here we give some lemmas which are necessary in the proof of Theorem~\ref{th1}.

\setcounter{theorem}{0}
\begin{lemma} \label{lema1} Let $0<\be<2$. If $f\in H_0^\beta$, then
$\Vert f\Vert_{\beta}^{2}$ given in (\ref{beta}) is
well defined and
\[
\Vert f\Vert_{\beta}^{2} = c^{-1}_{\beta,d }
\int_{\R^d}\bigl|\mathcal{F}f(\xi)\bigr|^{2}|\xi
|^{-\beta-d} \,d\xi\geq0,\vadjust{\goodbreak}
\]
where $\mathcal{F}f(\xi)$ denotes the Fourier transform of $f$, and
\[
c_{\beta, d }= \int_{\R^d} \bigl(1-\cos(x\cdot\xi)\bigr)
|\xi|^{-\beta
-d} \,d\xi>0
\]
is independent of $x$ if $|x|=1$.
\end{lemma}
\begin{pf}
For any $x\in S^{d-1}$ and any $d\times d$ orthogonal matrix $Q$,
the change of variable $\xi=Q\eta$ yields
\[
\int_{\R^d} \bigl(1-\cos(x\cdot\xi) \bigr) |
\xi|^{-\beta-d} \,d\xi =\int_{\R^d} \bigl(1-\cos\bigl(
\bigl(Q^Tx\bigr)\cdot\eta\bigr) \bigr) |\eta|^{-\beta-d} \,d
\eta>0.
\]
This shows that $\int_{\R^d}  (1-\cos(x\cdot\xi) )
|\xi|^{-\beta-d}  \,d\xi$ depends only on $|x|$. The substitution
$\xi=\frac1{|x|}\eta$ gives us $\int_{\R^d}  (1-\cos(x\cdot
\xi) ) |\xi|^{-\beta-d}  \,d\xi=c_{\be, d} |x|^{\be}$. Then an
elementary result from Fourier analysis \cite{stein} yields
\begin{eqnarray*}
\Vert f\Vert_{\beta}^{2} 
&=& c^{-1}_{\beta,d }
\int_{\R^{2d}}f(x)f(y) \biggl( \int_{\R^d}
\bigl(e^{\iota(x-y)\cdot\xi} -1 \bigr) |\xi|^{-\beta-d}\,d\xi \biggr) \,dx \,dy
\\
&=& c^{-1}_{\beta,d }\int_{\R^d}\bigl|
\mathcal{F}f(\xi)\bigr|^{2}|\xi |^{-\beta-d} \,d\xi\geq0.
\end{eqnarray*}
\upqed\end{pf}

\begin{lemma} \label{lema11} Assume that $1-H<Hd<1$. Let $X$ be a
$d$-dimensional centered normal random vector with covariance matrix
$\sigma^2I$. Then, for any $n\in\N$ and $y\in\R^d$, there exists a
constant $c$ depending only on $H$ and $d$ such that
\[
\int^{\infty}_{0} u^{-Hd} \E \biggl|\exp \biggl(
\iota\frac{y\cdot
X}{n^Hu^H} \biggr)-1 \biggr| \,du\leq c n^{Hd-1}|y|^{{1}/{H}-d}.
\]
\end{lemma}
\begin{pf} It suffices to show the above inequality when $y\neq0$.
Making the change of variable $v=|y|^{-{1}/{H}}nu$ gives
\begin{eqnarray*}
&&
\int^{\infty}_{0} u^{-Hd} \E \biggl|\exp
\biggl(\iota\frac{y\cdot
X}{n^Hu^H} \biggr)-1 \biggr| \,du
\\
&&\qquad=n^{Hd-1}|y|^{{1}/{H}-d} \int^{\infty}_{0}
v^{-Hd} \E \biggl|\exp \biggl(\iota\frac{y\cdot X}{|y| v^H} \biggr)-1 \biggr| \,dv
\\
&&\qquad\leq n^{Hd-1}|y|^{{1}/{H}-d} \int^{\infty}_{0}
v^{-Hd} \bigl(2\wedge v^{-H}\E|X| \bigr) \,dv
\\
&&\qquad= c n^{Hd-1}|y|^{{1}/{H}-d}.
\end{eqnarray*}
\upqed\end{pf}

\begin{lemma} \label{lema3}
Assume $1-H\leq Hd<1$ and $0\leq\beta\leq1$ with $H\beta<1-Hd$.
Then, for all $n\in\N$, there exists a constant $c$ independent of
$n$ such that
%
\setcounter{equation}{0}
\begin{eqnarray}
\label{ineq1}
&&\int_{0}^{ns}u^{-Hd-H\beta} \,du
\int_{\R^d}e^{-\kappa
|\xi|^{2}(s-{u}/{n})^{2H}}\bigl|1-e^{\iota{\xi\cdot
y}/{n^{H}}}\bigr| \,d\xi\nonumber\\[-8pt]\\[-8pt]
&&\qquad\leq c
n^{-H\beta} s^{-Hd-H\beta}|y|^{{1}/{H}-d}\nonumber
\end{eqnarray}
and
%
\begin{eqnarray}
\label{ineq2}\quad &&
n^{-Hd-H\beta}\int_{0}^{ns}
\biggl(s-\frac{u}{n} \biggr)^{-Hd-H\beta} \,du
\int_{\R^d}e^{-\kappa|\xi|^{2}({u}/{n})^{2H}}\bigl|1-e^{\iota
{\xi
\cdot y}/{n^{H}}}\bigr|
\,d\xi
\nonumber\\[-8pt]\\[-8pt]
&&\qquad \leq c n^{-H\beta} s^{-Hd-H\beta}|y|^{
{1}/{H}-d},
\nonumber
\end{eqnarray}
where $\kappa$ is a positive constant.
\end{lemma}

\begin{pf} We first note that (\ref{ineq1}) follows easily from
(\ref{ineq2}) by the change of variable. So, it suffices to prove
(\ref{ineq2}). Denote the left-hand side of (\ref{ineq2}) by $I$. We
then have
\begin{eqnarray*}
I & = & n^{-Hd-H\beta}\int_{0}^{ns} \biggl(s-
\frac{u}{n} \biggr)^{-Hd-H\beta
} \,du\int_{\R^d}e^{-\kappa|\xi| ^{2}(
{u}/{n})^{2H}}\bigl|1-e^{\iota{\xi
\cdot y}/{n^{H}}}\bigr|
\,d\xi
\\
& = & n^{-H\beta} \int_{0}^{ns}\biggl(s-
\frac{u}{n}\biggr)^{-Hd-H\beta}u^{-Hd}\int_{\R
^d}e^{-\kappa
|x|^{2}}
\bigl|1-e^{\iota{x\cdot y}/{u^{H}}} \bigr| \,dx \,du,
\end{eqnarray*}
where we made the change of variable $\xi(\frac{u}{n})^H=x$.

By Lemma \ref{lema11}, we obtain
\begin{eqnarray*}
&& \int_{0}^{{ns}/{2}}\biggl(s-\frac{u}{n}
\biggr)^{-Hd-H\beta}u^{-Hd}\int_{\R^d}e^{-\kappa|x|^{2}}
\bigl|1-e^{\iota{x\cdot y}/{u^{H}}}\bigr| \,dx \,du
\\
&&\qquad \leq c_1 s^{-Hd-H\beta}\int_{0}^{\infty}u^{-Hd}
\int_{\R
^d}e^{-\kappa|x|^{2}} \bigl|1-e^{\iota{x\cdot y}/{u^{H}}}\bigr| \,dx \,du
\\
&&\qquad = c_2 s^{-Hd-H\beta}|y|^{{1}/{H}-d}.
\end{eqnarray*}

On the other hand,
\begin{eqnarray*}
&& \int^{ns}_{{ns}/{2}}\biggl(s-\frac{u}{n}
\biggr)^{-Hd-H\beta}u^{-Hd}\int_{\R^d}
e^{-\kappa|x|^{2}} \bigl|1-e^{\iota{x\cdot y}/{u^{H}}}\bigr| \,dx \,du
\\
&&\qquad \leq\int^{ns}_{{ns}/{2}}\biggl(s-\frac{u}{n}
\biggr)^{-Hd-H\beta} u^{-Hd-H\beta_1}\int_{\R^d}e^{-\kappa|x|^{2}}|x
\cdot y|^{\beta
_1} \,dx \,du
\\
&&\qquad \leq c_3 (ns) ^{-Hd-H\beta_1} | y|^{\beta_1} \int
^{ns}_{{ns}/{2}}\biggl(s-\frac{u}{n}
\biggr)^{-Hd-H\beta}\,du
\\
&&\qquad \leq c_4 n^{1-Hd-H\beta_1}s^{1-H\beta-2Hd-H\beta
_1}|y|^{\beta_1},
\end{eqnarray*}
where $\beta_1$ can be any constant in $[0,1]$, and we used $H\beta
<1-Hd$ in the last inequality. Our result follows by choosing $\beta
_1=\frac{1}{H}-d$. This is possible because $1-H\leq Hd<1$.
\end{pf}

\begin{lemma} \label{lema4} For $n\in\N$, we assume that $1-H< Hd<1$
and $0<K<\frac{ns}{2}$. Then there
exists a constant $c$ independent of $n$ and $K$ such that
%
\begin{eqnarray}
\label{ineq3}
&&\int^{{ns}/{2}}_{K}u^{-Hd} \,du
\int_{\R^d}e^{-\kappa|\xi|
^{2}(s-{u}/{n})^{2H}}\bigl|1-e^{\iota{\xi\cdot y}/{n^{H}}}\bigr| \,d\xi \nonumber\\[-8pt]\\[-8pt]
&&\qquad\leq c
n^{1-Hd-H\beta}s^{1-2Hd-H\beta}|y|^{\beta}\nonumber
\end{eqnarray}
and
%
\begin{eqnarray}
\label{ineq4} && n^{-Hd}\int_{K}^{{ns}/{2}}
\biggl(s-\frac{u}{n} \biggr)^{-Hd} \,du\int_{\R^d}e^{-\kappa|\xi|
^{2}({u}/{n})^{2H}}\bigl|1-e^{\iota{\xi\cdot y}/{n^{H}}}\bigr|
\,d\xi
\nonumber\\[-8pt]\\[-8pt]
&&\qquad \leq c s^{-Hd} K^{1-Hd-H\beta} |y|^{\beta},
\nonumber
\end{eqnarray}
where $\kappa$ and $\beta$ are positive constants with
$1-Hd< H\beta<H\wedge(2-2Hd)$.
\end{lemma}
\begin{pf}Inequality (\ref{ineq3}) follows easily from the proof of
Lemma \ref{lema3}.
We shall show (\ref{ineq4}). Denote the left-hand side of
(\ref{ineq4}) by $I$. Then we have
\begin{eqnarray*}
I & = & n^{1-Hd}\int_{K/n}^{s/2}(s-u)^{-Hd}
\,du\int_{\R^d}e^{-\kappa
|\xi|^{2}u^{2H}}\bigl|1-e^{\iota{\xi\cdot y}/{n^{H}}}\bigr| \,d\xi
\\
& \leq & c |y|^{\beta} n^{1-Hd-H\beta} \int_{K/n}^{s/2}(s-u)^{-Hd}
u^{-Hd-H\beta} \,du.
\end{eqnarray*}
Since $\frac{K}{n}<\frac{s}{2}$ and $1-Hd< H\beta$, we have
\begin{eqnarray*}
\int_{K/n}^{s/2}(s-u)^{-Hd}
u^{-Hd-H\beta} \,du &\leq & c s^{-Hd}\int_{K/n}^{s/2}
u^{-Hd-H\beta} \,du
\\
& \leq & c s^{-Hd} \biggl(\frac{K}{n} \biggr)^{1-Hd-H\beta}.
\end{eqnarray*}
Therefore,
\[
I \leq c s^{-Hd} K^{1-Hd-H\beta} |y|^{\beta}.
\]
This proves the lemma.
\end{pf}

\begin{lemma} \label{lema5} For any $K>0$ and $n\in\N$, there exist
constants $c_1$ and $c_2$ independent of $n$ such that
%
\begin{equation}
\label{ineq5} \int^{K\wedge nu}_0 v^{-Hd}
\biggl(u-\frac{v}{n} \biggr)^{-Hd} \,dv \leq c_1K^{1-Hd}u^{-Hd}
\end{equation}
and
%
\begin{equation}\label{ineq6}
\int^{u}_0 v^{-Hd} (u-v
)^{-Hd} \,dv=c_2 u^{1-2Hd}.
\end{equation}
\end{lemma}
\begin{pf} Inequality (\ref{ineq6}) follows easily from
\[
\int^{u}_0 v^{-Hd} (u-v
)^{-Hd} \,dv=u^{1-2Hd}\int^{1}_0
w^{-Hd} (1-w )^{-Hd} \,dw.
\]
We only need to show (\ref{ineq5}). Notice that
\[
\int^{K\wedge nu}_0 v^{-Hd} \biggl(u-
\frac{v}{n} \biggr)^{-Hd} \,dv = u^{-Hd}(nu)^{1-Hd}
\int^{{K}/({nu})\wedge1}_0 v^{-Hd} (1-v
)^{-Hd} \,dv.
\]
If $nu\leq2K$, then
\[
\int^{{K}/({nu})\wedge1}_0 v^{-Hd} (1-v
)^{-Hd} \,dv \leq\int^{1}_0
v^{-Hd} (1-v )^{-Hd} \,dv.
\]

If $nu> 2K$, then
\[
\int^{{K}/({nu})\wedge1}_0 v^{-Hd} (1-v
)^{-Hd} \,dv \leq2 \int^{{K}/({nu})}_0
v^{-Hd} \,dv \leq\frac{2}{1-Hd} \biggl(\frac{K}{nu}
\biggr)^{1-Hd}.
\]

Therefore,
\[
\int^{K\wedge nu}_0 v^{-Hd} \biggl(u-
\frac{v}{n} \biggr)^{-Hd} \,dv \leq c_3
K^{1-Hd} u^{-Hd}.
\]
\upqed\end{pf}
\end{appendix}




\printaddresses

\end{document}